\documentclass[12pt]{amsart}

\usepackage{fullpage}
\usepackage{amsfonts,amssymb}
\usepackage{enumerate}
\usepackage{verbatim}
\usepackage{graphicx}
\usepackage{tikz-cd}
\usetikzlibrary{matrix}

\newcommand{\C}{\mathbb{C}}
\newcommand{\N}{\mathbb{N}}
\newcommand{\R}{\mathbb{R}}

\newcommand{\Z}{\mathbb{Z}}

\newcommand{\calD}{{\mathcal {D}}}

\newcommand{\calS}{{\mathcal {S}}}

\newcommand{\supp}{\operatorname{supp}}

\newcommand{\chip}{{\rho}}

\newcommand{\ov}{\overline}

\newcommand{\ve}{\varepsilon}
\newcommand{\vp}{\varphi}

\newcommand{\lan}{\langle}
\newcommand{\ran}{\rangle}

\newcommand{\B}{\mathbf B}
\newcommand{\bb}{\mathbf b}
\newcommand{\F}{\mathbf F}
\newcommand{\f}{\mathbf f}

\newcommand{\dist}{\operatorname{dist}}

\def\Koniec{\hbox to\hsize{\hfil$\diamond$}}

\def\ep{\varepsilon}

\def\1{\mathbf 1}

\newtheorem{theorem}{Theorem}[section]
\newtheorem{corollary}[theorem]{Corollary}
\newtheorem{lemma}[theorem]{Lemma}
\newtheorem{proposition}[theorem]{Proposition}

\theoremstyle{definition}
\newtheorem{definition}[theorem]{Definition}
\theoremstyle{remark}
\newtheorem{remark}[theorem]{Remark}

\numberwithin{equation}{section}

\begin{document}

\title{Parseval wavelet frames on Riemannian manifold}

\author{Marcin Bownik}
\address{Department of Mathematics, University of Oregon, Eugene, OR 97403--1222, USA}
\email{mbownik@uoregon.edu}

\author{Karol Dziedziul}
\address{
Faculty of Applied Mathematics,
Gda\'nsk University of Technology,
ul. G. Narutowicza 11/12,
80-952 Gda\'nsk, Poland}

\email{karol.dziedziul@pg.edu.pl}

\author{Anna Kamont}
\address{
Institute of Mathematics, Polish Academy of Sciences, ul. Abrahama 18, 80--825 Sopot, Poland}

\email{anna.kamont@impan.pl}

\keywords{Riemannian manifold, Hestenes operator, smooth decomposition of identity, wavelet frame, Triebel-Lizorkin space}

\thanks{The first author was partially supported by the NSF grant DMS-1956395. The authors are grateful to Isaak Pesenson for useful comments on Lemma \ref{male}.}

\subjclass[2000]{42C40, 46E30, 46E35, 58C35}

\date{\today}

\begin{abstract}
We construct Parseval wavelet frames in $L^2(M)$ for a general Riemannian manifold $M$ and we show the existence of wavelet unconditional frames in $L^p(M)$ for $1 < p <\infty$. This is made possible thanks to smooth orthogonal projection decomposition of the identity operator on $L^2(M)$, which was recently proven by the authors in \cite{BDK}. We also show a characterization of Triebel-Lizorkin $\F_{p,q}^s(M)$ and Besov $\B_{p,q}^s(M)$ spaces on compact manifolds in terms of magnitudes of coefficients of Parseval wavelet frames. We achieve this by showing that Hestenes operators are bounded on manifolds $M$ with bounded geometry. 
\end{abstract}

\maketitle

\section{Introduction}

The goal of this paper is to construct Parseval wavelet frames on Riemannian manifolds. This area dates back to the pioneering work of Ciesielski and Figiel \cite{CF1, CF2, CF3} who have constructed spline bases for Sobolev and Besov spaces on compact $C^\infty$ manifolds, see also \cite{FW}.
Ciesielski-Figiel decomposition of manifolds into cubes was subsequently used in the construction of wavelets on compact manifolds by Dahmen and Schneider \cite{DaSch} and by Kunoth and Sahner \cite{ks}.
Geller and Mayeli \cite{GM1} have constructed nearly tight frames on smooth compact oriented Riemannian manifold $M$ (without boundary) using Laplace-Beltrami operator on $L^2(M)$.
In a subsequent paper \cite{GM2} they have obtained a characterization of Besov spaces on a smooth compact oriented Riemannian manifold, for the full range of indices using smooth, nearly tight frames constructed in \cite{GM1}.
Geller and Pesenson \cite{GP} have constructed band-limited localized Parseval frames for Besov spaces on compact homogeneous manifolds.
Pesenson  has constructed nearly Parseval frames on noncompact symmetric spaces \cite{Pe2} and Parseval frames on sub-Riemannian compact homogeneous manifolds \cite{Pe3}.
Coulhon, Kerkyacharian, Petrushev \cite{CKP} have developed band limited well-localized frames in the general setting of Dirichlet spaces which includes complete Riemannian manifolds with Ricci curvature bounded from below and satisfying the volume doubling property. For 
 a survey on frames on Riemannian manifolds with bounded curvature and their applications to the analysis of function spaces see \cite{FHP}.

In this paper we improve upon these results by showing the existence of smooth Parseval wavelet frames on arbitrary Riemannian manifold $M$. Hence, we eliminate compactness assumption on $M$ needed in the work of Geller et al. \cite{GM1, GM2, GP} or Ricci curvature assumptions and volume doubling property needed in \cite{CKP}, and at the same time improve the construction of nearly tight frames to that of Parseval (tight) wavelet frames on $M$. This construction is made possible thanks to smooth orthogonal projection decomposition of the identity operator on $M$, which is an operator analogue of omnipresent smooth partition of unity subordinate to an open cover $\mathcal U$ of $M$, recently shown by the authors in \cite{BDK}. Our smooth orthogonal decomposition leads naturally to a decomposition of $L^2(M)$ as orthogonal subspaces consisting of functions localized on elements of an open and precompact cover $\mathcal U$. This enables the transfer of local Parseval wavelet frames from the Euclidean space to the manifold $M$ using geodesic maps. The resulting wavelet system, which consists of $C^\infty$ functions localized on geodesic balls, is a Parseval frame in $L^2(M)$. This construction extends to $L^p$ spaces  and yields unconditional dual wavelet frames in $L^p(M)$ for the entire range $1<p<\infty$. This is made possible by the extension of the above mentioned result in \cite{BDK} which yields a decomposition of the identity operator $\mathbf I$ on $L^p(M)$ as a sum of smooth projections $P_U$, which are mutually disjoint
\[
\sum_{U \in \mathcal U} P_U = \mathbf I, \qquad\text{where } P_U\circ P_{U'} =0 \text{ for } U\ne U' \in \mathcal U.
\]

In the case the manifold $M$ is compact we show a characterization of Triebel-Lizorkin and Besov spaces in terms of magnitudes of coefficients of Parseval wavelet frames. Our main theorem is inspired by a result due Triebel \cite{Tr5}, who has shown a characterization of Triebel-Lizorkin and Besov spaces by wavelets on compact manifolds. We improve upon his result in two directions. In contrast to \cite{Tr5}, our characterization allows the smoothness parameter $m$ to take the value $\infty$. Moreover, we employ a single wavelet system, which is used both in analysis and synthesis transforms. Since our wavelet system constitutes a Parseval frame in $L^2(M)$, it automatically yields a reproducing formula.

We achieve this result by proving the boundedness of Hestenes operators on Triebel-Lizorkin spaces on manifolds with bounded geometry. The study of function spaces on manifolds with bounded geometry was initiated by Triebel \cite{Tr1, Tr2}. More precisely, it is assumed that $M$ is a connected complete Riemannian manifold with positive injectivity radius and bounded geometry.
The theory of Triebel-Lizorkin and Besov spaces on such manifolds was further developed by Triebel \cite{Tr4}, Skrzypczak \cite{Sk},  and Gro\ss e and Schneider \cite{Grosse}. Our boundedness result is an extension of analogous result for Sobolev spaces shown in \cite{BDK}. A prototype of this result is due Triebel \cite{Tr4} who showed the boundedness of composition with a global diffeomorphism on Triebel-Lizorkin spaces on $\R^d$. 
We extend his result from the setting of $\R^d$ to the class of Hestenes operators on manifolds with bounded geometry. The proof uses a theorem due to Palais on an extension of local diffeomorphisms and results of Triebel \cite{Tr4} on boundedness of multipliers and diffeomorphisms on $\F^s_{p,q}$ spaces.  

The paper is organized as follows. In Section \ref{S2} we review necessary facts on manifolds $M$ with bounded geometry, results about Hestenes operators, and the definition of Triebel-Lizorkin spaces on $M$. In Section \ref{S3} we show that Hestenes operators are bounded on Triebel-Lizorkin  $\F_{p,q}^s(M)$ and Besov $\B_{p,q}^s(M)$ spaces. In Section \ref{S4} we construct smooth local Parseval frames on $\R^d$ using Daubechies and Meyer wavelets. 
In Section \ref{S5} we construct Parseval wavelet frames in $L^2(M)$ for a general Riemannian manifold $M$ and we show the existence of wavelet unconditional frames in $L^p(M)$ for $1 < p <\infty$. In Section \ref{S6} we show a characterization of Triebel-Lizorkin and Besov spaces on compact manifolds in terms of magnitudes of coefficients of Parseval wavelet frames constructed in the previous section. Finally, technical results characterizing wavelet coefficients of local distributions in Triebel-Lizorkin space $\F^s_{p,q}(\R^d)$ are shown in Section \ref{SL}.

\section{Preliminaries} \label{S2}

In this section we recall the necessary background on manifolds $M$ with bounded geometry such as a covering lemma by geodesic balls, the definition of Triebel-Lizorkin spaces on $M$, and facts about Hestenes operators and compositions of distributions with diffeomorphisms on manifolds. This is motivated by the fact that the definition of Triebel-Lizorkin spaces $\F_{p,q}^s=\F_{p,q}^s(M)$  requires the bounded geometry assumption on a Riemannian manifold $M$, see \cite[Section 7.2]{Tr4}.

\subsection{Bounded geometry}
Let $(M,g)$ be a $d$-dimensional connected complete Riemannian manifold with Riemannian metric tensor $g$. For any $x\in M$, the exponential geodesic map $\exp_x: T_x M \to M$ is a diffeomorphism of a ball $B(0, r) \subset T_xM$ of radius $r > 0$ with center $0$ and some neighborhood $\Omega_x(r)$ of $x$ in $M$. In fact, $\Omega_x(r)=\exp_x(B(0,r))$ is an open ball centered at $x$ and radius $r$ with respect to a geodesic distance on $M$. Denoting by $r_x$ the supremum of possible radii of such balls we define the injectivity radius of $M$ as $r_{inj} = \inf_{x\in M} r_x$. 
 We shall assume that a connected complete Riemannian manifold $M$ has {\it bounded geometry} \cite[Definition 1.1 in Appendix 1]{Sh} meaning that:
\begin{enumerate}
\item $r_{inj}>0$ and
\item every covariant derivative of the Riemann curvature tensor $R$ is bounded, that is, for any $k \in \N_0$, there exists a constant $C_k$ such that $|\nabla^k R| < C_k$. 
\end{enumerate}
The condition (2) can be equivalently formulated, see \cite[Section 7.2.1]{Tr4}, that there exist a positive constant $c$, and for every multi-index $\alpha$, positive constants $c_\alpha$, such that
\[
\det g \ge c \qquad\text{and}\qquad |D^\alpha g_{ij}| \le c_\alpha,
\]
in coordinates of every normal geodesic chart $(\Omega_x(r), i_x \circ \exp_x^{-1})$ for some fixed $0<r<r_{inj}$, where $i_x: T_xM \to \R^d$ is an isometric isomorphism (preserving inner products). The determinant $\det g$ is often abbreviated by $|g|$, see \cite{BDK, He}.

We have the following useful lemma about existence of covers by geodesic balls. A prototype of this lemma can be found in a monograph by Shubin \cite[Lemma 1.2 and 1.3 in Appendix 1]{Sh}, see also  \cite[Proposition 7.2.1]{Tr4}. The fact the multiplicity of the cover does not depend on the radius $r$ was observed by Skrzypczak \cite[Lemma 4]{Sk0}. A similar result can be found in \cite[Lemma 4.1]{FHP}, where a redundant assumption on local doubling property on $M$ was made.

\begin{lemma}\label{male}  Suppose a Riemannian manifold $M$ has bounded geometry. Then, for any $0 < r < r_{inj}/2$, there exists a set 
of points $\{x_j\}$ in $M$ (at most countable)  such that:
\begin{enumerate}[(i)]
\item the balls $\Omega_{x_j}(r/4)$ are disjoint,
\item the balls $\Omega_{x_j}(r/2)$ form a cover of $M$, and
\item for any $l\ge 1$ such that $r l < r_{inj}/2$, the multiplicity of the cover by the balls $\Omega_{x_j} (rl)$ is at most $N(l)$, where the constant $N(l)$ depends only on $l$ and a manifold $M$.
\end{enumerate}

Consequently, there exists a smooth partition of unity $\{\alpha_j\}$ corresponding to the open cover $\{\Omega_{x_j}(r)\}$,
\begin{equation}\label{pu}
\alpha_j\in C^\infty(M),
\qquad 
0 \le \alpha_j \le 1, \qquad \sum_j \alpha_j =1, \qquad \supp \alpha_j \subset  \Omega_{x_j}(r),
\end{equation}
such that for any mult-index $\alpha$, there exists a constant $b_\alpha$, satisfying
\begin{equation}\label{dr}
|D^\alpha(\alpha_j \circ \exp_{x_j} \circ (i_x)^{-1})| \le b_\alpha \qquad\text{for all }j\text{ and  } x\in B(0,r) \subset T_{x_j}M .
\end{equation}
In addition, for a fixed point $x\in M$, there exist $\{x_j\}$ and $\{\alpha_j\}$ satisfying:
\begin{equation}\label{px}
x=x_{j'} \quad\text{for some $j'$}\qquad\text{and}\qquad \alpha_{j'}=1\text{ on }\Omega_{x_{j'}}(r/2).
\end{equation}
\end{lemma}

\begin{proof}
Properties (i)--(iii) are a consequence of the proof of \cite[Lemma 1.2]{Sh} and \cite[Lemma 4]{Sk0}. We include details for the sake of completeness.
Take $\epsilon_0=r_{inj}/4$, and hence $3\epsilon_0 < r_{inj}$. Then, for any $r<2\epsilon_0$, we choose a maximal set of disjoint balls $\Omega_{x_j}(r/4)$, for some set 
of points $\{x_j\}$ in $M$. By  \cite[Lemma 1.2]{Sh} the balls $\Omega_{x_j}(r/2)$ form a cover of $M$. Similarly as in the proof of \cite[Lemma 4]{Sk0}, if $rl < r_{inj}/2$, then the multiplicity of the cover by the balls $\Omega_{x_j} (rl )$ is at most $N(l,r)$, 
\begin{equation}\label{N}
N(l,r)= (\sup_{y\in M} \operatorname{vol} \Omega_{y}(r(l+1/4) )(  \inf_{x\in M} \operatorname{vol} \Omega_x(r/4))^{-1}.
\end{equation}
We claim that there exists a constant $C>0$ such that
\begin{equation}\label{bv}
C^{-1} s^d \le \operatorname{vol} \Omega_{x}( s) \le C s^d \qquad\text{for all }x\in M, \ s<3\epsilon_0.
\end{equation}
Indeed, since $M$ has bounded geometry, there exists a constant $c>0$ such that
\[
c^{-1} \le \det g (y) \le c \qquad \text{for all }y\in M,
\]
where $\det g(y)$ denotes the determinant of the matrix whose elements are the components of $g$ in normal geodesic coordinates of a local chart $(\Omega_x(s), i_x \circ \exp_x^{-1})$ such that $y\in \Omega_x(s)$ and $s<3\epsilon_0=3r_{inj}/4$.  Observe that
\[
\operatorname{vol} \Omega_{x}( s) = \int_{B(0,s)}\sqrt{ \det g } \circ \exp_x \circ i_{x}^{-1} d\lambda,
\] 
where $\lambda$ denotes the Lebesgue measure on a ball $B(0,s) \subset \R^d$. Hence, the claim \eqref{bv} follows. By \eqref{N} and \eqref{bv} we have $N(l,r)  \le (4l+1)^dC^2$. Hence, $N(l,r)$ is independent of $r$.

Finally, the existence of a partition of unity satisfying \eqref{dr} is a standard fact, see \cite[Lemma 1.3]{Sh} and \cite[Proposition 7.2.1]{Tr4}. To show the additional part of Lemma \ref{male} we take a smooth function $\eta \in C^\infty(M)$ such that $0\le \eta \le 1$, $\eta=1$ on $\Omega_{x_{j'}}(r/2)$, and $\supp \eta \subset \Omega_{x_{j'}}(r)$. We define another smooth partition of unity $\{\tilde \alpha_j \}$ by
\[
\tilde \alpha_j =
\begin{cases} \alpha_j (1-\eta) & j \ne j',
\\
\alpha_{j'}+\eta(1 -\alpha_{j'}) & j=j'.
\end{cases}
\]
It is immediate that $\{\tilde \alpha_j \}$ satisfies \eqref{pu}. Next we observe that $\tilde \alpha_j =\alpha_j$ for all, but finitely many $j$. To show the analogue of \eqref{dr} for functions $\tilde \alpha_j$ we apply the product formula and we use the fact the support of $\eta$ is compact.
\end{proof}

\subsection{Distributions on {$M$}}
Before defining Triebel-Lizorkin spaces we recall basic definitions of distributions on a smooth Riemannian manifold $M$. We do not need to assume that $M$ has bounded geometry as we only need to know that $M$ is complete to have well-defined exponential geodesic maps (completeness assumption can be avoided if we use more general local charts).

Let $\calD(M)$ be the space of test functions consisting of all compactly supported complex-valued $C^\infty$ functions on $M$. Define the space of distribution $\calD'(M)$ as the space of linear functionals on $\calD(M)$. By \cite[Section 6.3]{Hor} a distribution in $\calD'(M)$ can be identified with the collection of distribution densities indexed by an atlas in $M$ and satisfying certain consistency identity \cite[formula (6.3.4)]{Hor}. We will illustrate how this identification works for distributions which are given as an integration against a locally integrable function.

Let $\nu$ be a Riemannian measure. A locally integrable function $f\in L^1_{loc}(M)$ defines a distribution in $ \calD'(M)$, which is customarily also denoted by $f$,
\[
f(\vp) = \int_M f(u) \vp(u) d\nu(u) \qquad\text{for }\vp \in \calD(M).
\]
 For $x\in M$ we consider a local geodesic chart $(\Omega_x(r),\kappa)$, where  $r=r_x$, $\kappa=\kappa_x=  i_x \circ \exp_x^{-1}$. Then the corresponding family of distribution densities indexed by $\kappa$ is given by
\[
\begin{aligned}
f_\kappa(\phi)& =f(\phi\circ \kappa)=\int_M f(u) \phi(\kappa(u)) d\nu(u)
\\
&=\int_{B(0,r)}  f(\kappa^{-1}(u)) \phi(u) \sqrt{ \det g_\kappa(\kappa^{-1}(u))} du \qquad\text{for }\phi \in \calD(B(0,r)),
\end{aligned}
\]
where $ \det g_\kappa$ denotes the determinant of the matrix whose elements are components of $g$ 
in coordinates of a chart $\kappa$. 
Then we can make an identification of $f_\kappa$ with a function
\begin{equation}\label{fku}
f_\kappa(u)=f(\kappa^{-1}(u)) \sqrt{\det g_\kappa(\kappa^{-1}(u)) } \qquad u\in B(0,r).
\end{equation}
Take two geodesic charts  $(\Omega_x(r),\kappa)$ and $(\Omega_{x'}(r'),\kappa')$ such that $\Omega_x(r)\cap \Omega_{x'}(r')\neq \emptyset$. Let $\psi= \kappa \circ (\kappa')^{-1}$. By \eqref{fku} we have
\begin{equation}\label{LALA}
\frac{f_{\kappa'}(u)}{\sqrt{\det g_{\kappa'}((\kappa')^{-1}(u)) }} =\frac{f_\kappa(\psi(u))}{\sqrt{\det g_\kappa((\kappa')^{-1}(u)) }}
\qquad\text{for } u\in \kappa'(\Omega_x(r) \cap \Omega_{x'}(r')).
\end{equation}
By the chain rule, see \cite[p. 120]{Ch2}, we have
\[
\sqrt{\det g_{\kappa'}(p) } = |\det \nabla \psi(\kappa'(p))|  \sqrt{\det g_{\kappa}(p) }
\qquad\text{for }p\in \Omega_x(r) \cap \Omega_{x'}(r').
\]
Hence, by \eqref{LALA} we obtain the consistency identity \cite[formula (6.3.4)]{Hor} 
\begin{equation}\label{consi}
f_{\kappa'}(u) = |\det \nabla \psi(u) | f_{\kappa}(\psi(u))
\qquad\text{for }u \in \kappa'(\Omega_x(r) \cap \Omega_{x'}(r')).
\end{equation}
Conversely, given a family of integrable functions $\{f_\kappa\}$ satisfying \eqref{consi} we deduce \eqref{LALA}. Applying \eqref{LALA} for $u=\kappa'(p)$ leads to a locally integrable function $f\in L^1_{loc}(M)$ given by
\[
f(p)=\frac{f_\kappa(\kappa(p))}{\sqrt{\det g_\kappa(p) }}= \frac{f_{\kappa'}(\kappa'(p))}{\sqrt{\det g_{\kappa'}(p) }} 
\qquad p\in \Omega_x(r) \cap \Omega_{x'}(r').
\]

Next we define a composition of distribution with a diffeomorphism \cite[Theorem 6.1.2 and Theorem 6.3.4]{Hor}.

\begin{definition}\label{dd}
Let $V \subset M$ and $V' \subset M'$ be open subsets of Riemannian manifolds $M$ and $M'$, respectively. Suppose that $\Phi: V \to V'$ is a $C^\infty$ diffeomorphism and $f\in  \calD'(V')$. Define $f\circ \Phi$ as a distribution in $\calD'(V)$ by 
\[
(f \circ \Phi)(\phi) = f((\phi \circ \Phi^{-1}) | \det \nabla \Phi^{-1}|)
\qquad\text{for }\phi \in \calD(V),
\]
where $|\det \nabla \Phi^{-1}|$ denotes the Jacobian determinant of the differential $\nabla\Phi^{-1}$ acting between tangent spaces of $M'$ and $M$.
\end{definition}

The following lemma shows that the above definition coincides with the usual composition when a distribution is a function.

\begin{lemma}\label{ddd}
Let  $\Phi: V \to V'$ is a $C^\infty$ diffeomorphism between open subsets $V$ and $V'$ of Riemannian manifolds $M$ and $M'$ with Riemannian measures $\nu$ and $\nu'$, respectively. Suppose that $f\in L^1_{loc}(V')$. Then treating $f$ as a distribution in $\calD'(V')$, which is given as an integration of $f$ against $\nu'$, the composition $f \circ \Phi$ is a distribution in $\calD'(V)$, which is given as an integration of the usual composition $f \circ \Phi$ against $\nu$.
\end{lemma}

\begin{proof}
Take any $\phi \in \calD(V)$. Then by Definition \ref{dd} and the change of variables formula on Riemannian manifold \cite[Theorem I.3.4]{Ch} we have
\[
\begin{aligned}
(f \circ \Phi)(\phi) = f((\phi \circ \Phi^{-1}) | \det \nabla \Phi^{-1}|) &=
\int_{V'} f(x) (\phi \circ \Phi^{-1})(x) | \det \nabla \Phi^{-1}(x)| d\nu'(x)
\\
& = \int_{V} f(\Phi(x)) \phi (x) d\nu(x).
\end{aligned}
\]
Hence, the distribution $f\circ \Phi$ coincides with a locally integrable function $f \circ \Phi \in L^1_{loc}(V)$.
\end{proof}

\subsection{Triebel-Lizorkin spaces}

We adapt the following definition of Triebel-Lizorkin spaces on Riemannian manifolds \cite[Definition 7.2.2]{Tr4}. Note that additionally we need to assume that $r<r_{inj}/8$, see \cite[Remark 7.2.1/2]{Tr4}.
 
 \begin{definition}\label{tlm}
 Let $M$ be a connected complete Riemannian manifold $M$ with bounded geometry.  Let $s \in \R$ and let $0 < p < \infty$ and $0<q \le \infty$. Then,
\begin{equation}\label{tlm0}
\F^s_{p,q}(M)=\{ f\in \calD'(M): \left(\sum_{j=1}^\infty \|\alpha_j f \circ \exp_{x_j} \circ i_{x_j}^{-1} ||_{\F^s_{p,q}(\R^d)}^p \right)^{1/p} <\infty \}.
\end{equation}
\end{definition}
Note that we interpret
$ \alpha_j f \circ (\exp_{x_j} \circ i_{x_j}^{-1})$ as a composition of a distribution $\alpha_j f$ on $M$ with a diffeomorphism $\exp_{x_j} \circ i_{x_j}^{-1}  : B(0,r) \to \Omega_{x_j}(r)$, see Definition \ref{dd}.  Hence, it is a compactly supported distribution in $\calD'(B(0,r))$, which can be extended by setting zero outside of $B(0,r) \subset \R^n$. Consequently,  we obtain a tempered distribution in $\calS'(\R^n)$ and the spaces  $\F_{p,q}^s(M)$ are defined locally using  $\F_{p,q}^s(\R^d)$ norm.
For the proof that this definition coincides in the case $M=\R^d$, see \cite[Proposition 7.2.2]{Tr4}. Moreover, the above definition is independent of the choice of the cover $\{\Omega_{x_j}(r)\}$ and the corresponding partition of unity $\{\alpha_j\}$ in Lemma \ref{male}, see  \cite[Theorem 7.2.3]{Tr4}.

\subsection{Hestenes operators}

Next we recall the definition of Hestenes operators \cite[Definition 1.1]{BDK} and their localization \cite[Definition 2.1]{BDK}.

\begin{definition}\label{H}
Let $M$ be a smooth connected Riemannian manifold (without boundary). Let $\Phi:V \to V'$ be a $C^\infty$ diffeomorphism between two open subsets $V, V' \subset M$. Let $\vp: M \to \R$ be a 
compactly supported 
$C^\infty$ function such that 
\[
\supp \vp =\ov{\{x\in M : \vp(x) \ne 0\}} \subset V.
\]
We define a simple $H$-operator $H_{\vp,\Phi,V}$ acting on a  function $f:M\to\C$ by
\begin{equation}\label{HeHe}
H_{\vp,\Phi,V}f(x) = \begin{cases} \vp(x) f(\Phi(x)) & x\in V
\\
0 & x\in M \setminus V.
\end{cases}
\end{equation}
Let $C_0(M)$ be the space of continuous complex-valued functions on $M$ that are vanishing at infinity, which is equipped with the supremum norm. Clearly, a simple $H$-operator induces a continuous linear map of the space $C_0(M)$ into itself. We define an $H$-operator to be a finite combination of such simple $H$-operators. The space of all $H$-operators is denoted by $\mathcal H(M)$.
\end{definition}

\begin{definition}\label{localized}
 We say that an operator $T \in \mathcal H(M)$ is {\it localized} on an open set $U \subset M$, if it is a finite combination of simple $H$-operators $H_{\vp,\Phi,V}$ satisfying $V\subset U$ and $\Phi(V) \subset U$.
\end{definition}

\begin{remark}\label{prec}
Note that every $H\in \mathcal H(M)$ has a representation which is localized on an open and precompact set since we assume that $\varphi$ in Definition \ref{H} is compactly supported.
\end{remark}

In \cite[Theorem 2.6]{BDK} we have shown that $H$-operators are bounded on $C^r(M)$ spaces and Sobolev spaces without any assumption on the geometry of $M$. 

\begin{theorem}\label{crm}
Suppose that $H \in \mathcal H(M)$ is localized on open and precompact set  $U \subset M$. Then, for any $r=0,1,\ldots$, the operator $H$ induces a bounded linear operator
\begin{align}\label{crm1}
&H: C^r(M) \to C^r(M), \qquad\text{where }r=0,1,\ldots,
\\
\label{crm2}
&H:W^r_p(M)\to W^r_p(M), \qquad\text{where }1\le p < \infty,\  r=0,1,\ldots.
\end{align}
\end{theorem}

Our goal is to extend Theorem \ref{crm} to Triebel-Lizorkin $F^s_{p,q}(M)$ spaces. To achieve this we need to define the action of Hestenes operators on distributions $\calD'(M)$. We shall use \cite[Lemma 2.12 and Corollary 2.13]{BDK} about adjoints of $H$-operators.

\begin{lemma}
\label{hest.adj}
Let $U \subset M$ be 
 an open and precompact subset of $M$. The following statements hold.
\begin{enumerate}[(i)]
\item
Let
$\Phi:V \to V'$ be a $C^\infty$ diffeomorphism between two open subsets $V, V' \subset U$  and let
$\vp: M \to \R$ be a $C^\infty$ be function such that 
\[
\supp \vp =\ov{\{x\in M : \vp(x) \ne 0\}} \subset V.
\]
The adjoint of the operator $H=H_{\vp,\Phi,V}$ is
  $ H^*=H_{\vp_1 ,\Phi^{-1},V'}$,
where
\[
\vp_1(y)=
\begin{cases}
\vp(\Phi^{-1}(y))\psi_1(y)& y\in V',
\\
0 & y\notin V'.
\end{cases}
\]
and $\psi_1$ is any $C^\infty(M)$ function such that
\[
\psi_1(y)=|\det \nabla\Phi^{-1}(y)|\qquad\text{for } y\in \Phi(\supp \vp),
\]
where $\nabla \Phi$ denotes the Jacobian linear map corresponding to $\Phi$ of tangent spaces of $M$.
That is, $H^*$  is a simple $H$-operator  localized on $U$ satisfying
\begin{equation}\label{adj}
\int_M H (f ) (x) g(x) d\nu(x) = \int_M f(y) H^* (g)(y) d\nu(y) \qquad\text{for all } f,g\in C_0(M),
\end{equation}
where $\nu$ is the Riemannian measure on $M$.
\item
Let $H\in  \mathcal H(M)$  be localized on open and precompact set $U$. That is, $H =  \sum_{i=1}^m H_i$, where each $H_i= H_{\vp_i, \Phi_i ,V_i}$ 
is a simple $H$-operator satisfying $V_i, \Phi_i(V_i) \subset U$.
Then, the adjoint $H^* = \sum_{i=1}^m (H_i)^* \in  \mathcal H(M)$ is  localized on  $U$ and \eqref{adj} holds. In particular, the action of $H^*$ on $C_0(M)$ does not depend 
on a representation of $H$ as a combination of simple $H$-operators.
\end{enumerate}
\end{lemma}

Note that the formula \eqref{adj} was initially shown in \cite{BDK} for $f,g\in C_c(M)$, but it also holds for $f,g\in C_0(M)$. Indeed, if $H=H_{\vp,\Phi,V}$ is a simple $H$-operator localized on $U$, then by choosing $\alpha \in C_c(M)$ such that $\alpha(x)=1$ for all $x\in U$, we have $H(f) = H(\alpha f)$ for all $f\in C_0(M)$ and $H^*(g)=H^*(\alpha g)$ for all $g\in  C_0(M)$. Hence, \eqref{adj} follows for simple $H$-operators and then for arbitrary $H$-operators.

\begin{definition}\label{dist}
For $f\in \calD'(M)$ define the action of an $H$-operator $H$ on $\calD(M)$ by
\[
Hf(\psi) =  f(H^*\psi) \qquad \text{for } \psi \in \calD(M).
\]
\end{definition}

Lemma \ref{hest.adj} implies that Hestenes operator $H^*\in \mathcal H(M)$  is well defined and continuous as a mapping $H^*:\calD(M) \to \calD(M)$.  Hence, $H$ is a well-defined mapping $H:\calD'(M)\to \calD'(M)$.
In case $M=\R^d$, we have also that $
H^*: \calS(\R^d) \to \calS(\R^d)$ is  well defined and continuous. Hence, the formula from Definition \ref{dist} extends to $f\in \calS'(\R^d)$ and  $H: \calS'(\R^d) \to \calS'(\R^d)$ is well-defined.

It is convenient to express Definition \ref{dist} in terms of a composition of a distribution in $\calD'(M)$ with a diffeomorphism.

\begin{lemma}\label{comp}
Let $H=H_{\vp,\Phi,V}$ be a simple $H$-operator on $M$. Let $f\in \calD'(M)$. Then,
\begin{equation}\label{comp0}
Hf(\psi) = [\vp (f \circ \Phi)](\psi) \qquad\text{for all }\psi \in \calD(M).
\end{equation}
\end{lemma}

\begin{proof}
Since $\Phi: V \to V'$ is a diffeomorphism and $f\in  \calD'(M)$, $f\circ \Phi$ is a distribution in $\calD'(V)$ by Definition \ref{dd} as a composition of a distribution with a diffeomorphism. In the case $f\in \calD'(M)$ is a function, then $f\circ \Phi$ is the usual composition by Lemma \ref{ddd}.  Take any $\psi \in \calD(M)$, $\supp \psi \subset V$. Applying respectively Definition \ref{dist}, Lemma \ref{hest.adj}, and  Definition \ref{dd} yields
\[
\begin{aligned}
Hf(\psi) =  f(H^*\psi) &= f( (\vp \circ \Phi^{-1}) |\det \nabla\Phi^{-1}| (\psi \circ \Phi^{-1}) )
\\
&
=f( ((\vp \psi) \circ \Phi^{-1}) |\det \nabla \Phi^{-1}|  )
= (f \circ \Phi)((\vp \psi)|_{V}).
\end{aligned}
\]
Since $\supp \vp \subset V$, we can extend a distribution $\vp (f \circ \Phi)$ to $\calD'(M)$ by setting zero outside of $V \subset M$. Hence, \eqref{comp0} follows.
\end{proof}

\section{Hestenes operators in Triebel-Lizorkin spaces}\label{S3}

In this section we show that Hestenes operators are bounded on Triebel-Lizorkin spaces $\F_{p,q}^s(M).$  A prototype of this result is due Triebel \cite[Theorem 4.2.2]{Tr4} who showed the boundedness of a global diffeomorphism on Triebel-Lizorkin spaces on $\R^d$.
We extend this result to the class of Hestenes operators from the setting of $\R^d$ to manifolds with bounded geometry. The proof uses a theorem due Palais on an extension of local diffeomorphisms and results of Triebel on boundedness of multipliers and diffeomorphisms on $\F^s_{p,q}$ spaces.

Recall that for  $\F_{p,q}^s(\R^d)$ space we have a topological embedding
\[
\calS(\R^d) \subset \F_{p,q}^s(\R^d) \subset \calS'(\R^d),
\]
see \cite[Remark 2.3.2/2]{Tr4}.
Hence,  $H$ is well-defined on $ \F_{p,q}^s(\R^d)$, $H: \F_{p,q}^s(\R^d)\to \calS'(\R^d)$.
 We also have the following 
topological embedding, see \cite[Theorem 7.4.2(i)]{Tr4}
\[
\calD(M) \subset \F_{p,q}^s(M) \subset \calD'(M).
\]
Hence,  $H$ is well-defined on $ \F_{p,q}^s(M)$, $H: \F_{p,q}^s(M)\to \calD'(M)$. Our goal is to show that $H$ is bounded on $ \F_{p,q}^s(M)$ spaces.

\begin{theorem}\label{hesten}
Let M be a connected complete
Riemannian manifold with bounded geometry. 
Let 
$0<p<\infty$, $0<q\leq \infty$, and $s\in \R$. 
Suppose that $H \in \mathcal H(M)$. 
Then the operator $H$ induces a bounded linear operator
\[
H:\F_{p,q}^s(M) \to \F_{p,q}^s(M).
\]
\end{theorem}

To prove this theorem we need  \cite[Theorem 4.2.2]{Tr4} about pointwise multipliers
\begin{theorem}\label{multiplier}
Let 
$0<p<\infty$, $0<q\leq \infty$, and $s\in \R$. If $m \in \N$ is sufficiently large, then there exists a constant $C_m$ such that for all $\vp \in C^m(\R^d)$ and $f\in F_{p,q}^s(\R^d)$
\[
\|\vp f ||_{\F_{p,q}^s(\R^d) }\leq C_m \sum_{|\alpha|\leq m} \|D^\alpha \vp||_{L^\infty(\R^d)} \|f ||_{\F_{p,q}^s(\R^d)}.
\]
\end{theorem}

It is known, see  remarks in the proof of  \cite[Theorem 4.2.2]{Tr4}, that  $m>s$ works
in the case $s>d/p$. In the case $-\infty<s\leq d/p$, Theorem \ref{multiplier} holds for $m>2d/p-s$.

For a smooth mapping $G:\R^d \to \R^d$ we denote Jacobian matrix  by $\nabla G(x)$.
We also need  \cite[Theorem 4.3.2]{Tr4}.
For simplicity we state this theorem for $C^\infty$-diffeomorphism.

\begin{theorem}\label{diffem}
Let 
$0<p<\infty$, $0<q\leq \infty$, and $s\in \R$. Let 
\[
G=(G_1,\ldots,G_d) :\R^d \to \R^d
\] be a $C^\infty$-diffeomorphism with all bounded derivatives, i.e.
for all multi-indices $\alpha$, there is $C(\alpha)$ such that
for all $x\in \R^d$ and  $j=1,\ldots,d$,
\[
|D^\alpha G_j(x)|\leq C(\alpha).
\]
Assume that there is a constant $c>0$ such that for all $x\in \R^d$
\[
|\det \nabla G(x) |>c.
\]
Then   $f \to f\circ G$ is an isomorphic mapping of $\F_{p,q}^s(\R^d)$ onto itself. 
\end{theorem}

We also need a theorem on an extension of local diffeomorphisms due to Palais
\cite[Theorem 5.5]{Pal1}, \cite{Pal2}.

\begin{lemma}\label{diffeo}
Let $V$, $V' \subset \R^d$ be  open sets and let $\Phi:V \to V'$ be a diffeomorphism. Then  for every $x\in V$, there is $\delta>0$ such that:
\begin{enumerate}[(i)]
\item $\Phi$ has an extension from a ball $B(x,\delta)\subset V$ to a global diffeomorphism 
\[
G=(G_{1},\ldots,G_{d}):\R^d \to \R^d,
\]
\item for all multi-indices $\alpha$, there exist constants $C(\alpha)$ such that
for all $y\in \R^d$ and all $j=1,\ldots,d$,
\[
|D^\alpha G_{j}(y)|\leq C(\alpha),
\]
\item there exists a constant $c>0$ such that for all $y\in \R^d$
\[
|\det \nabla G(y) |>c.
\]
\end{enumerate}
\end{lemma}

\begin{proof} It is convenient to follow the proof of Palais' theorem given by Lew \cite{Lew} in the setting of Banach spaces. Without loss of generality, by affine change of variables, we can assume that $x=0$, $\Phi(0)=0$, and $\nabla \Phi(0)={\mathbf I}$. Let $\eta \in C^\infty([0,\infty))$ be such that $\eta (u)=1$ for $u\in [0,1]$, $\eta(u)=0$ for $u\geq 2$, and $|\eta'(u)|<3/2$ for all $u$.
Following \cite{Lew}, the extension of a local diffeomorphism $\Phi:V \to V'$ in a neighborhood of $0$ is given for appropriate choice of $\delta>0$ by the formula
\begin{equation}\label{lew1}
G(x)=\lambda (x) \Phi(x)+(1-\lambda (x)) x,\qquad x\in \R^d,
\end{equation}
where $\lambda(x)=\eta(\|x\|/\delta)$. Part (i) is shown in \cite{Lew} as a consequence of the fact that for every $\ve>0$, there exists $\delta>0$, such that 
\begin{equation}\label{lew2}
|\nabla G(x) - \mathbf I | < \ve \qquad x\in \R^d.
\end{equation}
Part (ii) follows from \eqref{lew1}  since $G(x)=x$ for $|x|>2\delta$. Finally, (iii) follows immediately from \eqref{lew2}.
\end{proof}

Combining the above results yields the following theorem. 

 \begin{theorem}\label{composition}
 Let 
$0<p<\infty$, $0<q\leq \infty$, and $s\in \R$. 
Suppose that $H \in \mathcal H(\R^d)$. 
Then the operator $H$ induces a bounded linear operator
\[
H:\F_{p,q}^s(\R^d) \to \F_{p,q}^s(\R^d).
\]
\end{theorem}

\begin{proof}
 Since  an $H$-operator is a finite combination of simple $H$-operators, it is sufficient to prove the theorem for
a simple $H$-operator $H_{\vp,\Phi,V}$.
Without loss of generality, by Remark \ref{prec} we can assume that $H$ is localized on an open and precompact set  $U \subset \R^d$ satisfying
\[
\supp \vp \subset V \subset U, \quad \Phi(V) \subset U,
\]
where $\Phi:V \to V'$ is a $C^\infty$ diffeomorphism between two open subsets $V, V' \subset U$.
By Lemma \ref{diffeo} a local diffeomorphism $\Phi$  can be extended into a global diffeomorphism. That is, for every $x\in V$ there is a neighborhood (a ball)  $V_x\subset V$ and there is a constant $c_x>0$ and a $C^\infty$-diffeomorphism  with all bounded derivatives
\[
G_x :\R^d \to \R^d
\]
such that 
\[
\Phi(y)=G_x(y), \qquad y\in V_x
\]
and  
\[
|\det \nabla G_x(y) |>c_x \qquad y\in \R^d.
\]
Since $\supp \vp$ is a compact set, there is a finite set $\{x_j\}_{j=1}^N \subset V$ such that
\[
\supp \vp \subset \bigcup_{j=1}^N V_{x_j}.
\]
Thus for an open cover of $\R^d$ 
consisting of sets  $V_{x_j}, j=1,...,N$, and the set $\R^d\setminus \supp \vp$,
there is a smooth partition of unity $\alpha_j: \R^d \to \R$, $j=1,\ldots N+1$, such that for all $1\leq j\leq N+1$
\[
\supp \alpha_j \subset V_{x_j}
\]
and 
\[
\sum_{j=1}^N \alpha_j (y) = 1, \quad y\in \supp \vp.
\]
Note that for $f\in \calS(\R^d)$, and hence for $f\in \calS'(\R^d)$, we have
\[
H_{\vp,\Phi,V}f = \sum_{j=1}^N H_{\vp\alpha_j,\Phi,V} f.
\]
On the other hand, for all $1\leq j\leq N$,
\[
H_{\vp\alpha_j,\Phi,V}f=H_{\vp\alpha_j,G_{x_j},\R^d} f=\vp\alpha_j (f\circ G_{x_j}).
\]
This finishes the proof since $H_{\vp\alpha_j,G_{x_j},V}$ is a composition of operators satisfying assumptions of Theorem \ref{multiplier} and Theorem \ref{diffem}.
\end{proof}

We are ready to give the proof of Theorem \ref{hesten}.

\begin{proof}[Proof of Theorem \ref{hesten}] Without loss of generality, we assume that we have a simple $H$-operator $H_{\vp,\Phi,V}$ localized on a precompact set $U$. Let 
\begin{equation}\label{2r}
2r< \min\big\{ r_{inj}/8,\dist(\supp \vp,\partial V), \dist (\Phi(\supp \vp),\partial V') \big\}.
\end{equation}
Let $\{\alpha_j\}_j$ be
a partition of unity  subordinate to
uniformly locally finite cover of $M$ by a sequence of open
balls $\Omega_j=\Omega_{x_j}(r)$ as in Lemma \ref{male}.

Let $f\in \F_{p,q}^s(M)$.
We claim that it is sufficient to prove that for all $j$ such that $\Omega_j\cap \supp \vp \neq \emptyset$, there is  constant $C_j>0$ such that 
\begin{equation}\label{pe}
\|(\alpha_j \vp (f\circ \Phi))\circ \exp_{x_j} ||_{\F_{p,q}^s(\R^d)}
\leq
C_j \sum_{k\in I_j} \|(\alpha_k f )\circ \exp_{x_k}||_{\F_{p,q}^s(\R^d)},
\end{equation}
where
\[
I_j=\{k: \supp \alpha_j \cap \Phi^{-1}(\supp \alpha_k)\neq \emptyset \quad \& \quad \supp \vp \cap \Phi^{-1}(\supp \alpha_k) \neq \emptyset\}.
\]
For simplicity in \eqref{pe} we identify $\R^d$ with $T_{x_j}M$ and hence we omit the isometric isomorphism $i_{x_j}: T_{x_j}M \to \R^d$ as in Definition \ref{tlm}. 
Indeed, by Lemma \ref{comp} and \eqref{pe} we have
\begin{equation}\label{pe2}
\|(\alpha_j Hf ) \circ \exp_{x_j} ||_{\F_{p,q}^s(\R^d)}
\leq
C_j \sum_{k\in I_j} \|(\alpha_k f )\circ \exp_{x_k}||_{\F_{p,q}^s(\R^d)}.
\end{equation}
Since $\supp \vp$ is compact we have only finite number of $j$ such that
\begin{equation}\label{pes}
\supp \vp \cap \Omega_j \neq \emptyset
\end{equation}
and $I_j$ is finite. Otherwise, $\alpha_j Hf=0$. This is a consequence of the fact that $\Omega_j$ form a uniformly locally finite cover of $M$. Raising \eqref{pe2} to the power $p$ and summing over $j\in\N$ shows that $H$ is bounded on $ \F_{p,q}^s(M)$.

To prove \eqref{pe} note that for fixed $j$ we have
 \begin{equation}\label{pe3}
\begin{aligned}
(\alpha_j \vp (f\circ \Phi))\circ \exp_{x_j} &=
\sum_{k\in I_j}  (\alpha_j \vp(\alpha_k f)\circ \Phi))\circ \exp_{x_j}
\\
&=\sum_{k\in I_j} ((\alpha_j \vp)\circ \exp_{x_j})( (\alpha_k f)\circ \exp_{x_k} \circ \exp_{x_k}^{-1} \circ \Phi\circ \exp_{x_j}).
\end{aligned}
\end{equation}
By definition  $B(0,r)=\exp_{x_k}^{-1}(\Omega_k)$ is a ball of radius $r$ in $ \R^d$.
Note that by \eqref{2r} and \eqref{pes} we have $\Omega_j \subset V$ and $\Omega_k \subset V'$ for each $k\in I_j$. Hence,
$\Phi_k=\exp_{x_k}^{-1} \circ \Phi\circ \exp_{x_j}$ is a well-defined $C^\infty$-diffeomorphism
\[
\Phi_k : V_k=\exp_{x_j}^{-1}( \Omega_j \cap \Phi^{-1}(\Omega_k)) \to V_k'=\exp_{x_k}^{-1}( \Phi(\Omega_j) \cap \Omega_k).
\]
 Now we take a function
$\eta_k\in C^\infty(M)$ such that
\[
\eta_k(x)=1, \quad x\in \supp \alpha_j \cap \Phi^{-1}(\supp \alpha_k)
\]
and
\[
\supp \eta_k \subset  \Omega_j \cap \Phi^{-1}(\Omega_k).
\]
Since $k\in I_j$ then  by \eqref{pe3}
 \[
(\alpha_j \vp (f\circ \Phi))\circ \exp_{x_j}=
\sum_{k\in I_j} ((\eta_k \alpha_j \vp)\circ \exp_{x_j}) ((\alpha_k f)\circ \exp_{x_k} \circ \Phi_k)
\]
Defining new functions
\[
\tilde{\vp_k}=(\eta_k \alpha_j \vp)\circ \exp_{x_j}
\]
we have
 \begin{equation}\label{pe4}
(\alpha_j \vp (f\circ \Phi))\circ \exp_{x_j}=
\sum_{k\in I_j} \tilde{\vp_k}  ((\alpha_k f)\circ \exp_{x_k} \circ \Phi_k).
\end{equation}
The function $\tilde{\vp_k}$ and the distribution $\alpha_k f \circ \exp_{x_k}$ are defined locally on $B(0,r)$, but we can take their extension to $\R^d$ putting zero outside of $B(0,r)$. The presence of $\eta_k$ guarantees that $\supp \tilde{\vp_k} \subset V_k$ and it makes sense to consider a simple $H$-operator $H_{\tilde{\vp_k},\Phi_k,V_k}$, which is localized on $B(0,r) \subset \R^n$.
Now we apply Theorem \ref{composition} using Lemma \ref{comp}
\[
H_{\tilde{\vp_k},\Phi_k,V_k}(\alpha_k f \circ \exp_{x_k})= \tilde{\vp_k}  ((\alpha_k f)\circ \exp_{x_k} \circ \Phi_k).
\]
Hence, for some constant $C_k$, which is independent of $f$, we have
\[
\|
\tilde{\vp_k}  ((\alpha_k f)\circ \exp_{x_k} \circ \Phi_k)||_{\F_{p,q}^s(\R^d)}
\le C_k \|\alpha_k f \circ \exp_{x_k} ||_{\F_{p,q}^s(\R^d)}.
\]
Summing over $k\in I_j$ and using \eqref{pe4} yields \eqref{pe}, which completes the proof of Theorem \ref{hesten}.
\end{proof}
 
As a corollary of Theorem \ref{hesten} we can deduce the boundedness of Hestenes operators on Besov spaces. Recall that Besov spaces $\B_{p,q}^s(M)$ on manifolds $M$ with bounded geometry are introduced indirectly using real interpolation method of quasi-Banach spaces, see \cite[Section 7.3]{Tr4},
\begin{equation}\label{ri}
\B_{p,q}^s(M) = (\F_{p,p}^{s_0}(M), \F^{s_1}_{p,p}(M))_{\theta,q}, \qquad s=(1-\theta)s_0+\theta s_1, \ -\infty<s_0<s<s_1<\infty.
\end{equation}
By the functorial property of real interpolation we deduce that Hestenes operator are bounded on Besov $\B_{p,q}^s(M)$ spaces.

\begin{corollary}\label{hestenB}
Let M be a connected complete
Riemannian manifold with bounded geometry. 
Let $0<p<\infty$, $0<q\leq \infty$, and $s\in \R$. 
Then any operator $H \in \mathcal H(M)$ induces a bounded linear operator
$
H:\B_{p,q}^s(M) \to \B_{p,q}^s(M).
$
\end{corollary}

\section{Local Parseval Frames}\label{S4}

In this section we introduce the concept of smooth local Parseval frames on $\R^d$ and show their existence using Daubechies and Meyer wavelets. The main result of the section is Theorem \ref{localframe}, which is a local counterpart of the construction of wavelets on Euclidean space $\R^d$. It enables us to extend the construction of wavelets from the setting of $\R^d$ to manifolds and generalize a characterization of Triebel-Lizorkin and Besov spaces by wavelet coefficients. As in the classical case, we use compactly supported Daubechies wavelets when the smoothness parameter $m$ is finite and Meyer wavelets when $m=\infty$. The key part of the proof of Theorem \ref{localframe} is technical Lemma \ref{koniecprawie}, the proof of which is postponed to Section \ref{SL}.

We need to introduce the following notation.
Let $Q=(-1,1)^d$ and $\ep>0$. Let $Q_\ep=(-1-\ep,1+\ep)^d$. 
Let $E'=\{0,1\}^d$ be the vertices of the unit cube and let $E=E' \setminus \{0\} $ be the set of nonzero vertices.
For a fixed $j_0 \in \N_0$ we define
\[
E_{j}= \begin{cases} E' & \text{for }j=j_0,\\
E & \text{for } j>j_0.
\end{cases}
\]
Let $\chi_A$ be the characteristic function of a set $A$.

\begin{theorem}\label{localframe}
For given $m\in \N\cup \{\infty\}$, $j_0 \in \N_0$, and $\ep>0$, there exist a set of functions 
\begin{equation}\label{lf0}
\{
f^{\mathbf e}_{(j,k)}: j\geq j_0, k\in \Gamma_j, {\mathbf e}\in E_j
\},
\end{equation}
sets of indexes $\Gamma_j\subset \Z^d$, $j\geq j_0$, and a natural number $\lambda \ge 2$.
If $m$ is finite,
then
\[
\Gamma_j \subset \Lambda_j := \Z^d \cap [-2^{j-1}\lambda,2^{j-1}\lambda)^d.
\]
If $m=\infty$, then $\Gamma_j =\Z^d$.
 Define the family of functions $\{\chip_{j,k}: j \ge j_0, k\in \Gamma_j\}$ 
 as
 \begin{equation} \label{chip}
\chip_{j,k}(x)=\begin{cases}
2^{jd/2}\chi_I(2^j \lambda x-k)  & \text{if $m$ is finite, where } I=[0,1]^d,
\\ 2^{jd/2}\chi_I(2^j \lambda x-k') &  \text{if $m=\infty$, where $k' \in \Lambda_j$ and $ k-k' \in 2^j\lambda\Z^d$.}
\end{cases}
\end{equation}
 
 The family of functions \eqref{lf0} satisfies the following conditions:
 \begin{enumerate}[(i)]
   \item
 $f^{\mathbf e}_{(j,k)}\in C^m(\R^d)$ and
$\supp f^{\mathbf e}_{(j,k)}\subset Q_\ep$.
 
\item If $f\in L^2(\R^d)$ and $\supp f\subset Q$, then
\[
\|f||_2^2 =\sum_{j\geq j_0} \sum_{\mathbf e\in E_j} \sum_{k\in \Gamma_j} |\lan f,f^{\mathbf e}_{(j,k)} \ran |^2.
\]
 
\item Let $s\in \R$, $0<p<\infty$,  and $0<q\le \infty$. Suppose that
\begin{equation}\label{sm}
m> \max(s, \sigma_{p,q}-s), \qquad \sigma_{p,q}=d\max(1/p-1,1/q-1,0).
\end{equation}
If $f\in\F^s_{p,q}(\R^d)$ and $\supp f \subset Q$, then
\begin{equation}\label{ss}
f=\sum_{j\geq j_0} \sum_{\mathbf e\in E_j} \sum_{k\in \Gamma_j} \lan f, f^{\mathbf e}_{(j,k)} \ran f^{\mathbf e}_{(j,k)}
\end{equation}
with unconditional convergence in $\F^s_{p,q}$ if $q<\infty$ and in $\F^{s-\epsilon}_{p,q}$ spaces  for any $\epsilon>0$ if $q=\infty$. 

\item $\F^s_{p,q}$ norm is characterized by the magnitude of coefficients of functions \eqref{lf0}. That is, for any $f\in\F^s_{p,q}(\R^d)$ and $\supp f \subset Q$
we have
\begin{equation}\label{kwadratowa}
\| f||_{\F^s_{p,q}(\R^d)} \asymp
\bigg\| \bigg( \sum_{j\geq j_0} \sum_{\mathbf e\in E_j} \sum_{k\in \Gamma_{j}} \big(  2^{js} |\lan f,f^{\mathbf e}_{(j,k)}\ran|\chip_{j,k}\big)^q    \bigg)^{1/q} \bigg\|_p.
\end{equation}

 \end{enumerate}
\end{theorem}

\begin{definition}\label{lf}
A set of functions \eqref{lf0} which satisfies the conclusions of Theorem \ref{localframe} is said to be a local Parseval frame of smoothness $m$ and is denoted by $\mathcal W(m,j_0,\ep)$.
\end{definition}

We will give two proofs of Theorem \ref{localframe}. The first proof works only for finite smoothness $m$ using Daubechies wavelets. The second more general proof works for $m=\infty$ and uses Meyer wavelets.

\subsection{Daubechies multivariate wavelets}
We consider Daubechies multivariate wavelets following \cite{BD}. 
\begin{definition}\label{daubechies}
For a fixed $N\geq 2$, let ${}_N\phi$ be a univariate, compactly supported scaling function with support $\supp {}_N\phi=[0,2N-1]$
associated with the compactly supported, orthogonal univariate Daubechies wavelet ${}_N\psi$, see \cite[Section 6.4]{Dau}. In addition, we assume that $\supp {}_N\psi = [0,2N-1]$.
Let $\psi^0={}_N\phi$ and  $\psi^1={}_N\psi$.  For each $\mathbf e=(e_1,\ldots,e_d)\in E'$, define
\begin{equation}\label{mv}
\psi^{\mathbf e}(x)=\psi^{e_1}(x_1)\cdots \psi^{e_d}(x_d), \quad x=(x_1,\ldots,x_d)\in \R^d.
\end{equation}
For any $\mathbf e\in E'$, $j\in \Z$, and $k\in \Z^n$, we define Daubechies multivariate wavelet functions by
\begin{equation}\label{Daubechies}
\psi^{\mathbf e}_{j,k}(x)=2^{j d/2}\psi^{\mathbf e}(2^jx-k), \qquad x\in \R^d,
\end{equation}
\end{definition}

It is well-known that for any $j_0\geq 0$, a set $\{\psi^{\mathbf e}_{j,k}: j\geq j_0, e \in E_j, k\in \Z^d\}$ is an orthonormal basis of $L^2(\R^d)$. Moreover, it is also an unconditional basis of the Triebel-Lizorkin  space $\F^s_{p,q}(\R^d)$, $s\in \R$, $0< p<\infty$, $0<q<\infty$ for sufficiently large choice of $N$ depending on $s$, $p$, and $q$, see \cite[Theorem 1.20(ii)]{Tr5} and \cite[Theorem 3.5]{Trtom3} shown under more restrictive assumptions. More precisely,
$N=N(s,p,q)$ has to be such that $\psi^0={}_N\phi,  \psi^1={}_N\psi \in C^m(\R^d)$, where
\[
m> \max(s, \sigma_{p,q}-s), \qquad \sigma_{p,q}=d\max(1/p-1,1/q-1,0).
\]
Recall that the smoothness $m$ of Daubechies scaling function and wavelet ${}_N\phi$,  ${}_N\psi$ depends (roughly linearly) on $N$.

We shall illustrate the proof of Theorem \ref{localframe} when the parameter $j_0 \in \N_0$ depends on the smoothness $m$ and $\ep>0$.

\begin{definition}\label{poz}
 Let $j_0 \in \N_0$ be the smallest integer such that 
\begin{equation}\label{poziom}
(2N-1)2^{-j_0}< \ep/2.
\end{equation}
For $j\geq j_0$ define
\[
\Gamma_j=\{k\in \Z^d: \supp \psi^{\mathbf e}_{j,k} \subset Q_\ep\}.
\]
\end{definition}

\begin{proof}[Proof of Theorem \ref{localframe} for finite $m$]
Consider a Daubechies wavelet system of smoothness $m$ relative to the cube $Q$ and  $\ep>0$ defined by
\begin{equation}\label{los}
f^{\mathbf e}_{(j,k)} = \psi^{\mathbf e}_{j,k}, \qquad j\geq j_0, e \in E_j, k\in \Gamma_j.
\end{equation}
Observe that functions $\psi^{\mathbf e}$ given by \eqref{mv} satisfy $\supp \psi^{\mathbf e} = [0,2N-1]^d$. 
Hence, 
\[
\supp \psi^{\mathbf e}_{j,k} = 2^{-j}(k+[0,2N-1]^d).
\]
If this set
intersects the cube $Q=(-1,1)^d$ for some $j\ge j_0$, then by \eqref{poziom} we have $\supp \psi^{\mathbf e}_{j,k}  \subset Q_\ep$ and $k\in \Gamma_j$.

By Definition \ref{poz} the property (i) holds automatically. 
 Let $f\in L^2(\R^d)$ and $\supp f\subset Q$. If for some $j\ge j_0$, $k\in \Z^n$, and $\mathbf{e} \in E$, we have
$\lan f, \psi^{\mathbf e}_{j,k} \ran  \ne 0$, then $k\in \Gamma_j$. Since $\{\psi^{\mathbf e}_{j,k}: j\geq j_0, e \in E_j, k\in \Z^d\}$ is an orthonormal basis of $L^2(\R^d)$, we deduce (ii).

Let $s\in \R$, $0<p<\infty$,  and $0<q\le \infty$. Suppose that the smoothness $m$ satisfies \eqref{sm}. As before, if $f\in\F^s_{p,q}(\R^d)$ and $\supp f \subset Q$, then $\lan f, \psi^{\mathbf e}_{j,k} \ran  \ne 0$ implies that $k\in \Gamma_j$. By \cite[Theorem 1.20(ii)]{Tr5} 
\[
f=\sum_{j\geq j_0} \sum_{\mathbf e\in E_j} \sum_{k\in \Gamma_j} \lan f, \psi^{\mathbf e}_{j,k} \ran 
\psi^{\mathbf e}_{j,k}
\]  
with unconditional convergence in $\F^s_{p,q}$ norm if $q<\infty$; the pairing $\lan f, \psi^{\mathbf e}_{j,k} \ran $ makes sense by \cite[Remark 1.14]{Tr5}.
Since $ \lan f, \psi^{\mathbf e}_{j,k} \ran =0$ for $j\ge j_0$ and $k\not \in \Gamma_j$ we deduce (iii). If $q=\infty$ the above series converges locally in spaces $\F^{s-\epsilon}_{p,q}$ for any $\epsilon>0$. However, supports of $f$ and $f^{\mathbf e}_{(j,k)} $ are all contained in $Q_\ep$. Hence, the convergence in \eqref{ss} is in (global) $\F^{s-\epsilon}_{p,q}$ spaces for any $\epsilon>0$.
By \cite[Theorem 1.20(ii)]{Tr5}, the analysis transform
\[
\F^s_{p,q}(\R^d) \ni f \mapsto (\lan f, \psi^{\mathbf e}_{j,k} \ran )_{j\ge j_0,k\in \Z^d,\mathbf{e}
\in E_j} \in \f^s_{p,q}(\R^d)
\]
is an isomorphism, 
where $\f^s_{p,q}=\f^s_{p,q}(\R^d)$ is a discrete Triebel-Lizorkin space introduced by Frazier and Jawerth in \cite{FJ}. The $\f^s_{p,q}$ norm of a sequence $\boldsymbol s=(s_{j,k}^{\mathbf{e}})$ is given by
\begin{equation}\label{tld}
||\boldsymbol s||_{ \f^s_{p,q}}= 
\bigg\|
\bigg( \sum_{j\geq j_0} \sum_{\mathbf e\in E_j} \sum_{k\in \Z^d} \big(  2^{js} |s^{\mathbf e}_{j,k} |\chi_{j,k}\big)^q    \bigg)^{1/q} \bigg\|_p,
\end{equation}
where $\chi_{j,k}(x)=2^{jd/2}\chi_I(2^j x-k)$. Note that in \eqref{tld} we can replace functions $\chi_{j,k}$ by their scaled variants $x \mapsto 2^{jd/2}\chi_I(2^j \lambda x-k)$. 
Take any $f\in\F^s_{p,q}(\R^d)$ such that $\supp f \subset Q$.
Since $ \lan f, \psi^{\mathbf e}_{j,k} \ran =0$ for $j\ge j_0$ and $k\not \in \Gamma_j$, the norm equivalence \eqref{kwadratowa} follows.
\end{proof}

\subsection{Meyer multivariate wavelets} 

\begin{definition}\label{meyer}
Let  
   $\psi^0 \in \calS(\R^d)$  be the real-valued scaling function and let $\psi^1 \in \calS(\R^d)$ be  the associated real-valued Meyer wavelet, see \cite{HW, Me, Wo}. We define Meyer multivariate wavelets $\psi^{\mathbf e}_{j,k}$  in the same way as in \eqref{Daubechies}.
\end{definition}

It is well-known that for any $j_0\geq 0$, a set $\{\psi^{\mathbf e}_{j,k}: j\geq j_0, e \in E_j, k\in \Z^d\}$ is an orthonormal basis of $L^2(\R^d)$. Moreover, it is also an unconditional basis of the Triebel-Lizorkin  space $\F^s_{p,q}(\R^d)$ for all values of parameters $s\in \R$, $0< p<\infty$, see \cite[Theorem 3.12]{Trtom3}.
We shall now give the proof of Theorem  \ref{localframe} for $m=\infty$ using Meyer wavelets. For the sake of simplicity we shall assume that the scale parameter $j_0=0$; the general case follows by easy modifications.

\begin{proof}[Proof of Theorem \ref{localframe} for $m=\infty$]
Let $H$ be a Hestenes operator acting on functions on $\R^d$ such that:
\begin{enumerate}[(a)]
\item
$H$ is localized in $Q_\ep$; in particular, $H f(x)=0$ for all $f\in C_0(\R^d)$ and all $x\not\in Q_\ep$,
\item  $H f=f$ for all $f\in C_0(\R^d)$ such that $\supp f \subset Q$,
\item $H=H^*$ is an orthogonal projection on $L^2(\R^d)$.
\end{enumerate}
The existence of such operator in one dimension follows from the construction of Coifman and Meyer \cite{CM}, see \cite{AWW, HW}. The higher dimensional analogue is obtained by tensoring of one dimensional Hestenes operators, see \cite[Lemma 3.1]{BD}. That is, $H$ acts separately in each variable as one dimensional Hestenes operator. Since linear combinations of separable functions are dense in $L^2$ norm, we deduce that tensor product of two $H$-operators, which are orthogonal projections, is again an orthogonal projection. This shows the existence of an operator $H$ satisfying (a)--(c).

For $j\ge 0$ define $\Gamma_j=\Z^d$. Consider a Meyer wavelet system relative to the cube $Q$ and  $\ep>0$ defined by
     \[
 f^{\mathbf e}_{(j,k)}=H( \psi^{\mathbf e}_{j,k}), \qquad j\geq 0, e\in E_j, k\in \Z^d.
      \]  
 Properties of (i) and (ii) are  an immediate consequence of (a)-(c) and the fact that the multivariate Meyer wavelet system $\{\psi^{\mathbf e}_{j,k}: j\geq 0, e \in E_j, k\in \Z^d\}$ is an orthonormal basis of $L^2(\R^d)$.
 
 To show property (iii), take any $f\in\F^s_{p,q}(\R^d)$ such that $\supp f \subset Q$. By \cite[Theorem 3.12]{Trtom3} we have
\begin{equation}\label{ss3}
f=\sum_{j\geq 0} \sum_{\mathbf e\in E_j} \sum_{k\in \Gamma_j} \lan f, \psi^{\mathbf e}_{j,k} \ran \psi^{\mathbf e}_{j,k}
\end{equation}
with unconditional convergence in $\F^s_{p,q}$ if $q<\infty$ and locally in any $\F^{s-\epsilon}_{p,q}$ spaces  for $\epsilon>0$ if $q=\infty$. By property (b) we deduce that for $f\in \mathcal D'(\R^d)$, such that $\supp f \subset Q$, we have $Hf=f$. Applying the operator $H$ to both sides of \eqref{ss3} and using Theorem \ref{composition} yields the conclusion (iii). Since $\supp f^{\mathbf e}_{(j,k)}\subset Q_\ep$, the series \eqref{ss} converges (globally) in any $\F^{s-\epsilon}_{p,q}$ spaces  for $\epsilon>0$ if $q=\infty$.
 
The proof of (iv) is a consequence of Lemma \ref{koniecprawie}, whose proof is postponed till Section \ref{SL}.

\begin{lemma}\label{koniecprawie}
Let   $\{\psi^{\mathbf e}_{jk}: j\geq 0, k\in \Z^d, \mathbf e\in E_j\}$ be a multivariate Meyer wavelet  orthonormal basis of $L^2(\R^d)$.
Let $0<p<\infty$, $0<q\leq \infty$ and $s\in \R$. There exists a natural number $\lambda\geq 10$ such that for any $f\in \F^s_{p,q}(\R^d)$ with $\supp f\subset [-1,1]^d$ we have
\begin{equation}\label{periodic}
\| f||_{ \F^s_{p,q}(\R^d)}^p \asymp
\int_{\R^d} \bigg( \sum_{j\geq 0} \sum_{\mathbf e\in E_j} \sum_{ k\in \Lambda_{j}} \sum_{l\in \Z^d} \big( 2^{js}  \chi_{j,k}(x) |\lan f,\psi^{\mathbf e}_{j,k+2^jl\lambda }\ran| \big)^q \bigg)^{p/q} dx,
\end{equation}
where $\Lambda_j=\{k\in \Z^d: k/2^j\in [-\lambda/2,\lambda/2)^d\}$.
\end{lemma}

Take any $f\in\F^s_{p,q}(\R^d)$ with $\supp f \subset Q$.
For fixed $j\ge 0$ and $\mathbf e\in E_j$ we have
\[
\begin{aligned}
\sum_{ k\in \Lambda_{j}} \sum_{l\in \Z^d} \big( 2^{js}  \chi_{j,k}(x) |\lan f,\psi^{\mathbf e}_{j,k+2^jl\lambda}\ran| \big)^q
& = \sum_{k\in \Z^d} 
\big( 2^{js}  \chi_{j,n(k)}(x) |\lan f,\psi^{\mathbf e}_{j,k}\ran| \big)^q
\\
& = 
 \sum_{k\in \Z^d} 
\big( 2^{js}  \rho_{j,k}(x/\lambda) |\lan f,\psi^{\mathbf e}_{j,k}\ran| \big)^q,
\end{aligned}
\]
where $n(k)\in \Lambda_j$ is such that $n(k)-k \in 2^j \lambda \Z^d$. Since $f=Hf$ and $H=H^*$, we have $\lan f, f^{\mathbf e}_{(j,k)} \ran = \lan f, \psi^{\mathbf e}_{j,k}\ran$. 
Hence, Lemma \ref{koniecprawie} yields \eqref{kwadratowa} by the change of variables.
\end{proof}

\begin{remark}\label{synt}
Suppose that 
\[
\mathcal W(m,j_0,\ep)= \{ f^{\mathbf e}_{(j,k)}: j\geq j_0, k\in \Gamma_j, {\mathbf e}\in E_j \}
\]
is a local Parseval frame of smoothness $m$. Theorem \ref{localframe}(iv) shows the boundedness of the analysis transform defined on Triebel-Lizorkin space $\F^s_{p,q}(\R^d)$ for distributions $f$ satisfying $\supp f \subset Q$. To define synthesis operator we need to define a local version of  Triebel-Lizorkin sequence space $\f^{s,\mu}_{p,q}(\R^d)$ with an extra decay parameter $\mu>0$. Define $\f^{s,\mu}_{p,q}(\R^d)$ as the space of all sequences $\boldsymbol s=(s^{\mathbf e}_{(j,k)})$ with the quasi-norm
\begin{equation}\label{tlmu}
\begin{aligned}
||\boldsymbol s||_{\f^{s,\mu}_{p,q}} = 
\bigg\| \bigg( \sum_{j\geq j_0} \sum_{\mathbf e\in E_j} \sum_{k\in \Gamma_{j}} & \big(  2^{js} |s^{\mathbf e}_{(j,k)} |\chip_{j,k}\big)^q    \bigg)^{1/q}  \bigg\|_p\\
&+ \sup_{j\ge j_0,\  \mathbf e\in E_j, \ k\in \Gamma_{j} \setminus \Lambda_j} 2^{j\mu}(|2^{-j}k|_\infty+1)^\mu |s^{\mathbf e}_{(j,k)} |.
\end{aligned}
\end{equation}
If the smoothness $m$ is finite, then the second term is not present since $\Gamma_j \subset \Lambda_j$. Hence, the second term appears only when $m=\infty$ in which case $\Gamma_j=\Z^d$ and $\Lambda_j=\Z^d \cap [-2^{j-1}\lambda,2^{j-1}\lambda)^d$.
Then, for any $\mu>0$ the analysis operator (with respect to $\mathcal W(m,j_0,\ep)$) maps boundedly distributions $f\in \F^s_{p,q}(\R^d)$ with $\supp f \subset Q$ into $\f^{s,\mu}_{p,q}$ in light of Proposition \ref{p62}. Then for sufficiently large $\mu>0$, the synthesis operator
\begin{equation}\label{syn}
\boldsymbol s=(s^{\mathbf e}_{(j,k)}) \mapsto \sum_{j\geq j_0} \sum_{\mathbf e\in E_j} \sum_{k\in \Gamma_{j}}  
s^{\mathbf e}_{(j,k)} f^{\mathbf e}_{(j,k)}
\end{equation}
maps boundedly $\f^{s,\mu}_{p,q}$ into $\F^s_{p,q}(\R^d)$. To deduce this boundedness one needs to split the sum in \eqref{syn} over $k\in \Lambda_j$ and $k\in \Z^d \setminus \Lambda_j$. The former sum converges by the boundedness of synthesis operator from $\f^s_{p,q}(\R^d)$ to $\F^s_{p,q}(\R^d)$, see \cite[Theorem 3.12]{Trtom3}. The latter sum converges by the same argument as in the proof of Proposition \ref{klucz} for $\mu>\max(d/p,s+d/2)$.
\end{remark}

We have the following extension of Theorem \ref{localframe} to Besov spaces.

\begin{theorem}\label{localframeB}
Under the hypothesis of Theorem \ref{localframe}, in addition to (i)--(iv) the following conclusions hold:
\begin{enumerate}
\item[(v)]  Let $s\in \R$, $0<p< \infty$,  and $0<q\le \infty$.  Suppose that
\begin{equation}\label{smB}
m> \max(s, \sigma_{p}-s) , \qquad \sigma_{p}=d\max(1/p-1,0).
\end{equation}
If $f\in\B^s_{p,q}(\R^d)$ and $\supp f \subset Q$, then
\begin{equation}\label{ssB}
f=\sum_{j\geq j_0} \sum_{\mathbf e\in E_j} \sum_{k\in \Gamma_j} \lan f, f^{\mathbf e}_{(j,k)} \ran f^{\mathbf e}_{(j,k)}
\end{equation}
with unconditional convergence in $\B^s_{p,q}$ if $q<\infty$ and in $\B^{s-\epsilon}_{p,q}$ spaces  for any $\epsilon>0$ if $q=\infty$. 

\item[(vi)]  $\B^s_{p,q}$ norm is characterized by the magnitude of coefficients of functions \eqref{lf0}. That is, for any $f\in\B^s_{p,q}(\R^d)$ and $\supp f \subset Q$
we have
\begin{equation}\label{kwadratowaB}
\| f||_{\B^s_{p,q}(\R^d)} \asymp
\bigg( \sum_{j\geq j_0} 2^{j(s+d/2-d/p)q}  \sum_{\mathbf e\in E_j} \bigg(  \sum_{k\in \Gamma_{j}}  |\lan f,f^{\mathbf e}_{(j,k)}\ran|^p \bigg)^{q/p}    \bigg)^{1/q} .
\end{equation}
\end{enumerate}
\end{theorem}

\begin{proof}
If the smoothness parameter $m$ is finite, then let $\{\psi^{\mathbf e}_{j,k}\}$ be a multivariate Daubechies wavelet. By \cite[Theorem 1.20(i)]{Tr5} the analysis transform
\[
\B^s_{p,q}(\R^d) \ni f \mapsto (\lan f, \psi^{\mathbf e}_{j,k} \ran )_{j\ge j_0,k\in \Z^d,\mathbf{e}
\in E_j} \in \bb^s_{p,q}(\R^d)
\]
is an isomorphism, 
where $\bb^s_{p,q}=\bb^s_{p,q}(\R^d)$ is a discrete Besov space. The $\bb^s_{p,q}$ norm of a sequence $\boldsymbol s=(s_{j,k}^{\mathbf{e}})$ is given by
\begin{equation}\label{tldB}
||\boldsymbol s||_{ \bb^s_{p,q}}= 
\bigg( \sum_{j\geq j_0}  2^{j(s+d/2-d/p)q} \sum_{\mathbf e\in E_j} \bigg( \sum_{k\in \Z^d}   |s^{\mathbf e}_{j,k} |^p
\bigg)^{q/p}    \bigg)^{1/q}.
\end{equation}
Take any $f\in\B^s_{p,q}(\R^d)$ such that $\supp f \subset Q$.
Since $ \lan f, \psi^{\mathbf e}_{j,k} \ran =0$ for $j\ge j_0$ and $k\not \in \Gamma_j$, 
the formula \eqref{ssB} and
the norm equivalence \eqref{kwadratowaB} follow by the same argument as for Triebel-Lizorkin spaces.

If the smoothness parameter $m=\infty$, then we use multivariate Meyer wavelet instead. By \cite[Theorem 3.12(i)]{Trtom3}, the analysis transform
\[
\B^s_{p,q}(\R^d) \ni f \mapsto (\lan f, \psi^{\mathbf e}_{j,k} \ran )_{j\ge 0,k\in \Z^d,\mathbf{e}
\in E_j} \in \bb^s_{p,q}(\R^d)
\]
is an isomorphism and an analogue of formula \eqref{ss3} for Besov spaces holds. Then  \eqref{ssB} follows by the same argument as for Triebel-Lizorkin spaces. Finally, we deduce \eqref{kwadratowaB} using the isomorphism of analysis transform and the fact that
$\lan f, f^{\mathbf e}_{(j,k)} \ran = \lan f, \psi^{\mathbf e}_{j,k}\ran$.
\end{proof}

\section{Unconditional frames in $L^p(M)$}\label{S5}

In this section we combine Theorem \ref{localframe} on local Parseval frame and our earlier results \cite{BDK} on smooth decomposition of identity in $L^p(M)$ to construct unconditional frames in $L^p(M)$. It is worth emphasizing that our construction does not use any assumption on Riemannian manifold (such as completeness or bounded geometry). In particular, we show the existence of smooth Parseval wavelet frames in $L^2(M)$ on arbitrary Riemannian manifold $M$. This construction is made possible thanks to the following fundamental result \cite[Theorem 6.2]{BDK}.

\begin{theorem}\label{Main}
Let $M$ be a smooth connected Riemannian manifold (without boundary) and let $1\le  p < \infty$.
Suppose $\mathcal U$ is an open and precompact cover of $M$. Then, there exists $\{P_U\}_{U\in\mathcal U}$ a smooth decomposition of identity in $L^p(M)$, subordinate to $\mathcal U$. That is, the following conditions hold:
\begin{enumerate}[(i)]
\item family $\{P_U\}_{U\in \mathcal U}$ is locally finite, i.e., for any compact $K\subset M$, all but finitely many operators $P_{U}$ such that $U \cap K \ne \emptyset$, are zero, 
\item each $P_U \in \mathcal H(M)$ is localized on an open set $U \in\mathcal U$,
\item each $P_U: L^p(M) \to L^p(M)$ is a projection,
\item $P_U \circ P_{U'}=0$ for any $U \neq U' \in \mathcal U$,
\item $\sum_{U\in\mathcal U}P_U= \mathbf I$, where $\mathbf I$ is the identity in $L^p(M)$ and the convergence is unconditional in strong operator topology,
\item 
there exists a constant $C>0$ such that
\begin{equation}\label{unc0}
\frac 1{C} ||f||_{p}  \le \bigg( \sum_{U \in\mathcal U} ||P_U f||_p^p \bigg)^{1/p} \le C ||f||_p \qquad\text{for all }f\in L^p(M).
\end{equation}
\end{enumerate}
In the case $p=2$, the decomposition constant $C=1$ and each $P_U$, $U\in\mathcal U$, is an orthogonal projection on $L^2(M)$.
%
\end{theorem}

Recall that $M$ is $d$-dimensional Riemannian manifold.
For every $x\in M$ there exists $r=r(x)>0$ such that the exponential geodesic map $\exp_x$ is well defined  diffeomorphism of a ball $B(0, r) \subset T_xM$ of radius $r > 0$ with center $0$ and some precompact neighborhood $\Omega_x(r)$ of $x$ in $M$. 
For $x\in M$ we consider a local geodesic chart $(\Omega_x(r),\kappa)$, where  $r=r(x)$, $\kappa=\kappa_x=  i_x \circ \exp_x^{-1}$, and $i_x: T_xM \to \R^d$ is an isometric isomorphism. 
Define $T^p_x:  L^p(B(0, 3\sqrt{d})) \to L^p(\Omega_x(r))$ given by
\begin{equation}\label{dtp}
 T^p_x f(u)=
\bigg( \frac{3\sqrt{d}}{r} \bigg)^{d/p}
 \frac{f( \frac{3\sqrt{d}}{r} \kappa(u) ) }{|\det g_\kappa(u) |^{1/(2p)}}  
 \qquad \text{for } u \in \Omega_x(r),
\end{equation}
where $ \det g_\kappa$ denotes the determinant of the matrix whose elements are components of $g$ in coordinates of a chart $\kappa$. 

\begin{lemma}\label{iso}
 Let $1<p<\infty$. For each $x\in M$, the operator $T^p_x:  L^p(B(0, 3\sqrt{d})) \to L^p(\Omega_x(r))$ is an isometric isomorphism.
Moreover, we have the identity $(T^p_x)^{-1}=(T^{p'}_x)^*$, where $1/p+1/p'=1$. 
\end{lemma}

\begin{proof}
Take any $f\in  L^p(B(0, 3\sqrt{d}))$.  Then, by the definition of Riemannian measure $\nu$ and the change of variables we have
\[
\begin{aligned}
||T^p_x f||^p_p = \bigg( \frac{3\sqrt{d}}{r} \bigg)^d  \int_{\Omega_x(r)}  
 \frac{|f( \frac{3\sqrt{d}}{r} \kappa(u))|^p }{|\det g_\kappa(u) |^{1/2}} d\nu(u)
 & = \bigg(\frac{3\sqrt{d}}{r} \bigg)^d  \int_{B(0,r)}  \bigg |f \bigg( \frac{3\sqrt{d}}{r}u \bigg) \bigg|^p du 
 \\
 & =   \int_{B(0,3\sqrt{d} )}   |f(u) |^p du= ||f||^p_p. 
 \end{aligned}
\]
A similar calculation shows that for any $f\in  L^p(B(0, 3\sqrt{d}))$ and $h\in  L^{p'}(B(0, 3\sqrt{d}))$ we have
\[
\lan T^{p}_x f , T^{p'}_x h \ran =\int_{\Omega_x(r)} T^{p}_x f T^{p'}_x h d\nu = \int_{B(0,3\sqrt{d})}  f h dx = \lan f, h \ran.
\]
Take any $k\in L^p(\Omega_x(r))=(L^{p'}(\Omega_x(r))^*$. Then by the definition of adjoint for any $h\in  L^{p'}(B(0, 3\sqrt{d}))$ we have
\[
\lan (T^{p'}_x)^*k, h \ran = \lan k, T^{p'}_x h \ran  = \lan (T^p_x)^{-1} k, h \ran.
\]
Since $h$ is arbitrary we have $(T^p_x)^{-1}=(T^{p'}_x)^*$.
\end{proof}

We choose $0<\ep<1/2$ such that
\[
B(0,1)\subset Q=(-1,1)^d\subset Q_\ep \subset B(0,3\sqrt{d}).
\]
We take a local Parseval frame of smoothness $m\in \N \cup\{\infty\}$
\[
\mathcal W(m,j_0,\ep)= \{ f^{\mathbf e}_{(j,k)}: j\geq j_0, k\in \Gamma_j, {\mathbf e}\in E_j\}
\]
as in Theorem \ref{localframe}.

Next we transport a local Parseval frame $\mathcal W(m,j_0,\ep)$  to the manifold $M$ using operators $T^p_x$ and  $T^{p'}_x$, where $1/p+1/p'=1$.

\begin{lemma}\label{reconst}
 For any $f\in L^p(M)$, $1<p<\infty$, such that  $\supp f \subset \Omega_x(r/(3\sqrt{d}))$, $r=r(x)$, we have a reconstruction formula
 \begin{equation}\label{rec}
f= \sum_{j\geq j_0}\sum_{{\mathbf e}\in E_j} \sum_{k\in\Gamma_j}\lan f, T^{p'}_x f^{\mathbf e}_{(j,k)} \ran T^p_x f^{\mathbf e}_{(j,k)},
\end{equation}
with unconditional convergence in $L^p(M)$. Moreover, 
\begin{equation}\label{perio}
\| f||_{L^p(M)} 
\asymp
\bigg\| \bigg( \sum_{j\geq j_0} \sum_{\mathbf e\in E_j} \sum_{k\in \Gamma_{j}}   \big( |\lan f, T^{p'}_x (f^{\mathbf e}_{(j,k) }) \ran| T^p_x (\chip_{j ,k}) \big)^2    \bigg)^{1/2} \bigg\|_{L^p(M)},
\end{equation}
where $\chip_{j,k}$ are given by \eqref{chip}.
\end{lemma}

\begin{proof}
Since $1<p<\infty$, we can identify the Triebel-Lizorkin space $\F^0_{p,2}(\R^d) = L^p(\R^d)$.
We have 
\[
\supp (T^p_x)^{-1} f  \subset B(0, 1) \subset Q.
\]
Hence, by Theorem \ref{localframe} and Lemma \ref{iso} we have
 \[
(T^p_x)^{-1} f = \sum_{j\geq j_0}\sum_{{\mathbf e}\in E_j} \sum_{k\in\Gamma_j}\lan (T^p_x)^{-1} f , f^{\mathbf e}_{(j,k)} \ran f^{\mathbf e}_{(j,k)}
=\sum_{j\geq j_0}\sum_{{\mathbf e}\in E_j} \sum_{k\in\Gamma_j}\lan  f , T^{p'}_x f^{\mathbf e}_{(j,k)} \ran f^{\mathbf e}_{(j,k)}
,
\]  
with unconditional convergence in $L^p(\R^d)$. Applying $T^p_x$ to both sides yields the reconstruction formula \eqref{rec}. 

By Theorem \ref{localframe} and Lemma \ref{iso}, we have
\[
\begin{aligned}
\| (T^p_x)^{-1} f||_p \asymp
& 
\bigg\| \bigg( \sum_{j\geq j_0} \sum_{\mathbf e\in E_j} \sum_{k\in \Gamma_j}    \big( |\lan (T^p_x)^{-1} f,f^{\mathbf e}_{(j,k)}\ran|\chip_{j,k}\big)^2    \bigg)^{1/2} \bigg\|_p
\\
= &  \bigg\| \bigg( \sum_{j\geq j_0} \sum_{\mathbf e\in E_j} \sum_{k\in \Gamma_j}    \big( |\lan  f,T^{p'}_x(f^{\mathbf e}_{(j,k)}) \ran|\chip_{j,k}\big)^2    \bigg)^{1/2} \bigg\|_p.
\end{aligned}
\]
We claim that 
\begin{equation}\label{chi}
\supp \chip_{j,k} \subset [-1,1]^d \qquad\text{for } j\ge j_0,\  k\in \Gamma_j.
\end{equation}
Indeed, by \eqref{chip} there exists $k'\in \Lambda_j$ such that
\[
\supp \chip_{j,k} = \supp \chip_{j,k'} = (2^j \lambda)^{-1}([0,1]^d+k') \subset [0,2^{-j}\lambda]^d + [-1/2,1/2]^d \subset [-1,1]^d.
\]
Hence, we can apply the operator $T^p_x$ to functions $\chip_{j,k}$. Using \eqref{dtp} and Lemma \ref{iso} yields
\eqref{perio}.
\end{proof}

\begin{theorem}\label{pframes}
Let $M$ be a connected Riemannian manifold (without boundary) and $1<p<\infty$. 
Let $\mathcal W(m,j_0,\ep)$ be a local Parseval frame of smoothness $m\in \N \cup\{\infty\}$, $j_0 \in \N_0$, and $0<\ep<1/2$. Then, there exists at most countable subset $X\subset M$ and a collection of projections $P_{\Omega_x}$, $x\in X$, on $L^p(M)$ such that:
\begin{enumerate}[(i)]
\item
 for $f\in L^p(M)$, 
\[
f=\sum_{x\in X} \sum_{j\geq j_0}\sum_{{\mathbf e}\in E_j} \sum_{k\in\Gamma_j}\lan f, (P_{\Omega_x})^* T^{p'}_x f^{\mathbf e}_{(j,k)} \ran  P_{\Omega_x} T^p_x f^{\mathbf e}_{(j,k)},
\]  
with unconditional convergence in $L^p(M)$,
\item for $f\in L^{p'}(M)$, $1/p+1/p'=1$,
\[
f=\sum_{x\in X} \sum_{j\geq j_0}\sum_{{\mathbf e}\in E_j} \sum_{k\in\Gamma_j}\lan f, P_{\Omega_x} T^p_x f^{\mathbf e}_{(j,k)} \ran (P_{\Omega_x})^* T^{p'}_x f^{\mathbf e}_{(j,k)},
\]
with unconditional convergence in $L^{p'}(M)$,
\item
for any $f\in L^p(M)$ we have
\begin{equation}\label{cwp}
||f||_{L^p(M)}^p \asymp
\sum_{x\in X} \bigg\| \bigg( \sum_{j\geq j_0} \sum_{\mathbf e\in E_j} \sum_{k\in \Gamma_j}    ( |\lan f, (P_{\Omega_x})^*T^{p'}_x (f^{\mathbf e}_{(j,k) }) \ran| T^p_x (\chip_{j ,k}) \big)^2    \bigg)^{1/2} \bigg\|_{L^p(M)}^p.
\end{equation}
\end{enumerate}
\end{theorem}

\begin{proof}
Let $\mathcal U$ be an open precompact cover consisting of geodesic balls
\[
\mathcal U = \{ \Omega_x  :=\Omega_x(r(x)/(3\sqrt{d})):  x\in M \}.
\]
We apply Theorem \ref{Main} to the open cover $\mathcal U$ to obtain a smooth decomposition of identity $\{P_{\Omega_x}\}_{x\in M}$ in $L^p(M)$, subordinate to $\mathcal U$. By Theorem \ref{Main}(i) at most countably many projections $P_{\Omega_x}$ are non-zero. Hence, there exists at most countable subset $X \subset M$ such that $\{P_{\Omega_x}\}_{x\in X}$ is a smooth decomposition of identity in $L^p(M)$.
By Theorem \ref{Main}(v) for any $f\in L^p(M)$ we have
\[
f= \sum_{x\in X} P_{\Omega_x} f
\]
with unconditional convergence in $L^p(M)$. Applying \eqref{rec} for each function $ P_{\Omega_x} f$, using the fact that $P_{\Omega_x}$ is a projection, and summing over $x\in X$ yields (i). 
By \cite[Theorem 2.15]{BDK} the family $\{(P_{\Omega_x})^*\}_{x\in X}$ is a smooth decomposition of identity in $L^{p'}(M)$. Hence, the same argument yields (ii).
Finally, by Theorem \ref{Main}(vi) we have for any $f\in L^p(M)$,
 \[
 ||f ||^p_p   \asymp   \sum_{x\in X} ||P_{\Omega_x} f ||^{p}_p .
 \]
Applying \eqref{perio} to each function $P_{\Omega_x}f$ yields \eqref{cwp}.
\end{proof}

Let $\mathcal W^p(M)$ denote the wavelet system given by Theorem \ref{pframes}:
\begin{equation}\label{wm}
\mathcal W^p(M)=\{ P_{\Omega_x} T^p_x f^{\mathbf e}_{(j,k)}: x\in X, j\geq j_0, k\in \Gamma_j, {\mathbf e}\in E_j\},
\end{equation}
and its dual wavelet system
\begin{equation}\label{wmp}
\mathcal W^{p'}(M)=\{ (P_{\Omega_x})^{*} T^{p'}_x f^{\mathbf e}_{(j,k)}: x\in X, j\geq j_0, k\in \Gamma_j, {\mathbf e}\in E_j\}.
\end{equation}
Note that by \cite[Theorem 2.15]{BDK} the definition of the dual system \eqref{wmp} is consistent with the definition of the wavelet system \eqref{wm}.

As an immediate corollary of Theorem \ref{pframes} we deduce the fact that $\mathcal W^2(M)$ is a Parseval frame of $L^2(M)$.

\begin{corollary}\label{parseval}
For any $m\in \N\cup \{\infty\}$, the family 
$\mathcal W^2(M)$
is Parseval frame in $L^2(M)$ consisting of $C^m$ functions localized on geodesic balls $\Omega_x$, $x\in X$. That is,
\[
||f||^2_2= \sum_{x\in X} \sum_{j\geq j_0}\sum_{{\mathbf e}\in E_j} \sum_{k\in\Gamma_j} |\lan f, P_{\Omega_x} T^{2}_x f^{\mathbf e}_{(j,k)} \ran  |^2
\qquad\text{for all }f\in L^2(M).
\]
\end{corollary}

\begin{proof}
When $p=2$ we have $(P_{\Omega_x})^*=P_{\Omega_x}$ is an orthogonal projection on $L^2(M)$. By Theorem \ref{pframes}(i) we have 
\[
f=\sum_{x\in X} \sum_{j\geq j_0}\sum_{{\mathbf e}\in E_j} \sum_{k\in\Gamma_j}\lan f, P_{\Omega_x} T^{p}_x f^{\mathbf e}_{(j,k)} \ran  P_{\Omega_x} T^p_x f^{\mathbf e}_{(j,k)}
\qquad\text{for } f\in L^2(M),
\]  
with unconditional convergence in $L^2(M)$. Since $P_{\Omega_x}$ is an $H$-operator localized on $\Omega_x$ and operators $T^p_x$ preserve smoothness, we deduce the corollary.
\end{proof}

For general $1<p<\infty$, Theorem \ref{pframes} implies that the pair $(\mathcal W^p(M), \mathcal W^{p'}(M))$ is an unconditional frame of $L^p(M)$. The concept of a Banach frame was originally introduced by Gr\"ochenig \cite{Grochenig}, see also \cite[Definition 2.2]{Casazza}. We shall use the following definition of a (Schauder) frame \cite[Definition 2.2]{cdosz}. 

\begin{definition}
Let $B$ be a an infinite dimensional separable Banach space. Let $B'$ be the dual space of $B$. A sequence $(f_j, g_j)_{j\in \N}$ 
with $(f_j)_{j\in \N} \subset B$ and $(g_j)_{j\in \N} \subset B'$, is called a (Schauder) frame of $B$ if for every $f\in B$ we have
\[
f = \sum_{j\in \N} \lan g_j, f\ran f_j 
\]
with convergence in norm, i.e.,  $f=\lim_{n\to \infty} \sum_{j=1}^n  \lan g_j, f\ran f_j $. An unconditional frame of $B$ is a frame  $(f_j, g_j)_{j\in \N}$ 
 of $B$ for which the above series converges unconditionally.
\end{definition}

A frame in a Banach space can be equivalently characterized in terms of a space of scalar valued sequences, see \cite[Theorem 2.6]{Casazza}. In particular, we have the following proposition \cite[Proposition 2.4]{cdosz}.

\begin{proposition}
A sequence $(f_j, g_j)_{j\in \N}$ is an unconditional frame of $B$ if and only if the following conditions hold:
\begin{enumerate}[(i)]
\item there exists a Banach space $Z$ of scalar valued sequences  
such that coordinate vectors $(e_j)_{j\in \N}$ form an unconditional basis of $Z$ with  corresponding coordinate functionals $(e_j^*)_{j\in \N}$, 
\item there exist an isomorphic embedding $T: B \to Z$, and a surjection $S: Z \to B$, so that $S \circ T = {\mathbf I}_B$, $S(e_j)=f_j$ for $j\in \N$, and $T^*(e_i^*) =g_j$ for $j\in \N$ with $f_j \ne 0$.  
\end{enumerate}
\end{proposition}

The operator $T$ is often called an analysis transform, $S$ is a synthesis transform, and $Z$ is the sequence space of frame coefficients. We can reformulate Theorem \ref{pframes} in terms of Banach frames as follows. 

\begin{corollary}\label{cfp}
Let $M$ be a connected Riemannian manifold (without boundary) and $1<p<\infty$. 
Then the pair of dual wavelet systems  $(\mathcal W^p(M), \mathcal W^{p'}(M))$, given by Theorem \ref{pframes}, is an unconditional frame of  $L^p(M)$. \end{corollary}

The sequence space of frame coefficients is described via the formula \eqref{cwp} when the smoothness parameter $m$ is finite. If $m=\infty$, it is necessary to add an additional decay term as in Remark \ref{synt}, see also Remark \ref{synt2}.
In the case when $M$ has bounded geometry, we can improve this construction.

\begin{theorem}\label{pcof}
Let $M$ be a connected $d$-dimensional Riemannian manifold with bounded geometry and $1<p<\infty$. 
Then the dual wavelet system  $(\mathcal W^p(M), \mathcal W^{p'}(M))$ from Theorem \ref{pframes} can be chosen in such a way that there exist sets $\Omega_{j,k,x} \subset M$ satisfying
\begin{equation}\label{cwp2}
||f||_{L^p(M)} \asymp
 \bigg\|  \bigg( \sum_{x\in X} \sum_{j\geq j_0}\sum_{{\mathbf e} \in E_j} \sum_{k\in \Gamma_{j}}
 2^{jd}  |\lan f, (P_{\Omega_x})^*T^{p'}_x (f^{\mathbf e}_{(j,k) }) \ran|^2 \chi_{\Omega_{j,k,x}} \bigg)^{1/2} \bigg\|_{L^p(M)}.
\end{equation}
\end{theorem}

\begin{proof}
Since $M$ has positive injectivity radius, there exists $r_0<r_{inj}$ such that 
 the exponential geodesic map $\exp_x$ is well defined  diffeomorphism of a ball $B(0, r) \subset T_xM$ and $\Omega_x(r)$ with the same radius $r=r_0$ for all $x\in M$. By Lemma \ref{male} applied to $r'=r_0/(3\sqrt{d})<r_{inj}/2$ and $l=3\sqrt{d}/2$, there exists a set 
of points $X' \subset M$ (at most countable)  such that
 the family of balls $\mathcal U = \{\Omega_{x}(r'/2)\}_{x\in X'}$ is a cover of $M$, and the multiplicity of the cover $ \{\Omega_{x}(r'l)=\Omega_x(r_0/2)\}_{x\in X'}$
  is finite. Repeating the proof of Theorem \ref{pframes} for $\mathcal U$ yields the same conclusion with additional property that  $X\subset X'$. 
  In addition, we also have formula \eqref{cwp}.
 For $j\ge j_0$, $k\in \Gamma_{j}$, and $x\in X$ we define
 \[
 \Omega_{j,k,x} = (\kappa_x)^{-1}( r_0/(3\sqrt{d}) \supp \chip_{j,k} ).
 \]
 By \eqref{chi} we have 
 \begin{equation}\label{ino}
 \Omega_{j,k,x} \subset (\kappa_x)^{-1}( r_0/(3\sqrt{d}) [-1,1]^d) \subset \Omega_x(r_0/2).
 \end{equation}
 By \eqref{dtp} we have
 \[
|T^p_x (\chip_{j ,k})(u)| = 2^{jd/2}  \bigg( \frac{3\sqrt{d}}{r_0} \bigg)^{d/p}  \frac{\chi_{\Omega_{j,k,x}}(u)}{|\det g_\kappa(u) |^{1/(2p)}} 
 \qquad\text{for }u\in M.
 \]
 By the assumption of bounded geometry we have
 \begin{equation}\label{tro}
|T^p_x (\chip_{j ,k})(u)| \asymp 2^{jd/2} \chi_{\Omega_{j,k,x}}(u)
 \qquad\text{for }u\in M.
 \end{equation}
Hence, by \eqref{cwp} we have
\[
\begin{aligned}
||f||^p_p & \asymp
 \int_M \sum_{x\in X} \bigg( \sum_{j\geq j_0} \sum_{\mathbf e\in E_j} \sum_{k\in \Gamma_{j}}     2^{jd} |\lan f, (P_{\Omega_x})^*T^{p'}_x (f^{\mathbf e}_{(j,k) }) \ran|^2 \chi_{\Omega_{j,k,x}}(u)   \bigg)^{p/2} d\nu(u)
\\
& \asymp
 \int_M  \bigg( \sum_{x\in X} \sum_{j\geq j_0} \sum_{\mathbf e\in E_j} \sum_{k\in \Gamma_{j}}     2^{jd} |\lan f, (P_{\Omega_x})^*T^{p'}_x (f^{\mathbf e}_{(j,k) }) \ran|^2 \chi_{\Omega_{j,k,x}}(u)   \bigg)^{p/2} d\nu(u).
 \end{aligned}
\]
The last step follows from \eqref{ino}, the fact that the multiplicity of the cover $\{\Omega_{x}(r_0/2)\}_{x\in X}$ is finite, and the equivalence of finite dimensional $\ell^1$ and $\ell^{2/p}$ (quasi)-norms.
\end{proof}

Motivated by Theorem \ref{pcof} we give a definition of discrete Triebel-Lizorkin spaces on manifolds $M$ with bounded geometry.

\begin{definition}\label{fs}
Suppose that the manifold $M$ has bounded geometry. Let $\mathcal W^2(M)$ be a Parseval frame in $L^2(M)$ consisting of $C^m$ functions localized on geodesic balls $\Omega_x=\Omega_x(r_0/3\sqrt{d})$, $x\in X$, as in Corollary \ref{parseval}.
 Let $s \in \R$, $0 < p < \infty$, and $0<q \le \infty$. We define a discrete Triebel-Lizorkin space $\f^s_{p,q}=\f^s_{p,q}(M)$ as a set of sequences 
 \[
 \boldsymbol s= \{s_\psi\}_{\psi\in \mathcal W^2(M)}, \qquad
 \psi = P_{\Omega_x} T^2_x f^{\mathbf e}_{(j,k)}\text{ for } x\in X, j\geq j_0, k\in \Gamma_{j}, {\mathbf e}\in E_j,
 \]
such that
\[
\| \boldsymbol s \|_{\f^s_{p,q}}=
 \bigg\|  \bigg( \sum_{x\in X} \sum_{j\geq j_0}\sum_{{\mathbf e} \in E_j} \sum_{k\in \Gamma_{j}}
 2^{jq(s+d/2)}  |s_\psi|^q \chi_{\Omega_{j,k,x}} \bigg)^{1/q} \bigg\|_{ L^p(M)}
<\infty.
\]
\end{definition}

Note that when $M$ is a compact manifold, the set $X$ is necessarily finite and the above definition is similar to that given by Triebel \cite[Definition 5.7]{Tr5}.

\section{Parseval frames on compact manifolds}\label{S6}

In this section we show a characterization of Triebel-Lizorkin spaces on compact manifolds in terms of magnitudes of coefficients of Parseval wavelet frames constructed in the previous section. Our main theorem is inspired by a result due Triebel \cite[Theorem 5.9]{Tr5}, which we improve upon in two directions. In contrast to \cite{Tr5}, Theorem \ref{triebel} allows the smoothness parameter $m$ to take the value $\infty$. Moreover, it employs a single wavelet system $\mathcal W^2(M)$ for analysis and synthesis transforms, which constitutes a Parseval frame in $L^2(M)$ and it automatically yields a reproducing formula.

We start with the fundamental result about the decomposition of function spaces on compact manifolds, which is an extension of \cite[Theorem 7.1]{BDK} to the setting of Triebel-Lizorkin spaces.

 \begin{theorem}\label{rozklad}
Let $M$ be a smooth compact Riemannian manifold (without boundary).   Let $\mathcal F(M)=\F^s_{p,q}(M)$ be the Triebel-Lizorkin space, where $s\in \R$, $0<p<\infty$,  and $0<q\leq \infty$.
Let $\{P_U\}_{U\in \mathcal U}$ be a smooth orthogonal decomposition of identity in $L^2(M)$, which is subordinate to a finite open cover $\mathcal U$ of $M$. 
Then, we have a direct sum decomposition
\[
\mathcal F(M) = \bigoplus_{U \in \mathcal U} P_U( \mathcal F(M)),
\]
with the equivalence of norms
\[
||f||_{\mathcal F(M)} \asymp \sum_{U \in\mathcal U} ||P_U f||_{\mathcal F(M)} \qquad\text{for all }f\in \mathcal F(M).
\] 
 \end{theorem}
 
\begin{proof}
The proof of Theorem \ref{rozklad} employs Theorem \ref{hesten} and is shown in a similar way as in \cite[Theorem 6.1]{BD}. This is possible due to the fact that the number of projections $\{P_U\}_{U \in \mathcal  U}$ is finite and hence they are uniformly bounded on $\mathcal F(M)$.
That is, there exists a  constant $C>0$ such that 
\[
\|P_U f\|_{\mathcal F(M) }\leq C \|f\|_{\mathcal F(M)}
\qquad\text{for all } U \in \mathcal U, \ f\in \mathcal F(M)=\F^s_{p,q}(M).
\]
Since each $P_U$ is a projection, $P_U(\mathcal F(M))=\ker(P_U-\mathbf I)$ is a closed subspace of $\mathcal F(M)$. It remains to show that the operator $T$ defined by $Tf=(P_U f )_{U \in\mathcal U} $ is an isomorphism between $\mathcal F(M)$ and $\bigoplus_{U \in \mathcal U}  P_U(\mathcal F(M))$. Since $\{P_U\}_{U\in \mathcal U}$ is a smooth decomposition of identity in $L^2(M)$ we have
\[
f = \sum_{U\in\mathcal U} P_U f \qquad\text{for all }f\in \mathcal D(M).
\]
Hence, by Definition \ref{dist} and the fact that $(P_U)^*=P_U$ we have
\[
f = \sum_{U\in\mathcal U} P_U f \qquad\text{for all }f\in \mathcal D'(M).
\]
Hence, the operator $T$ is $1$-to-$1$. The operator $T$ is onto due to the fact that $P_U \circ P_V =0$ for $U \not = V \in \mathcal U$.
\end{proof}

Next we show  an analogue of Lemma \ref{reconst} for $\F^{s}_{p,q}(M)$ spaces.

\begin{lemma}\label{l52}
Let $M$ be a $d$-dimensional manifold with bounded geometry.
Let $s\in \R$, $0<p<\infty$,  $0<q\le \infty$. Let $m\in \N\cup \{\infty\}$ be such that
\begin{equation}\label{smp}
m> \max(s, \sigma_{p,q}-s), \qquad \sigma_{p,q}=d\max(1/p-1,1/q-1,0).
\end{equation}
Let $f\in \F^{s}_{p,q}(M)$ be such that $\supp f \subset \Omega_x(r/(3\sqrt{d}))$, where $x\in M$ and $0<r<r_{inj}/8$. 
Then, we have a reconstruction formula
 \begin{equation}\label{recs}
f= \sum_{j\geq j_0}\sum_{{\mathbf e}\in E_j} \sum_{k\in\Gamma_j}\lan f, T^{2}_x f^{\mathbf e}_{(j,k)} \ran T^2_x f^{\mathbf e}_{(j,k)},
\end{equation}
with unconditional convergence in $\F^s_{p,q}$ if $q<\infty$ and in $\F^{s-\epsilon}_{p,q}$ spaces  for any $\epsilon>0$ if $q=\infty$. 
Furthermore,
we have
\begin{equation}\label{perio2}
\| f||_{\F^s_{p,q}(M)}
\asymp
\bigg\| \bigg( \sum_{j\geq j_0} \sum_{\mathbf e\in E_j} \sum_{k\in \Gamma_{j}}  2^{jq(s+d/2)}   |\lan f, T^{2}_x (f^{\mathbf e}_{(j,k) }) \ran|^q \chi_{\Omega_{j ,k,x}}    \bigg)^{1/q}  \bigg\|_{L^p(M)}.
\end{equation}
\end{lemma}

\begin{proof}
In Definition \ref{tlm} of Triebel-Lizorkin spaces we have a freedom of choosing a partition of unity $\{\alpha_j\}$ described in Lemma \ref{male} with $r<r_{inj}/8$. We require that $\{\alpha_j\}$ satisfies \eqref{px} in addition to \eqref{pu} and \eqref{dr}. Consequently, the sum \eqref{tlm0} collapses to one term
\begin{equation}\label{kappa}
||f||_{\F^s_{p,q}(M)} \asymp \|\alpha_{j'} f \circ \exp_{x_{j'}} \circ i_{x_{j'}}^{-1} ||_{\F^s_{p,q}(\R^d)}
=  \|f \circ \exp_{x} \circ i_{x}^{-1} ||_{\F^s_{p,q}(\R^d)}.
\end{equation}
Let $\kappa=i_x \circ \exp_x^{-1}$. For $a>0$ define a dilation operator $\delta_a g(x)= a^{d/2} g(ax)$, where $g$ is a function defined on subset of $\R^d$. We can similarly define a dilation operator on distributions by
\[
\langle  \delta_a g, \phi \rangle = \langle g, \delta_{a^{-1}} \phi \rangle \qquad\text{for }\phi \in \mathcal D(\R^d).
\]
Since $\supp f \subset \Omega_x(r/(3\sqrt{d}))$, by choosing $a=r/(3\sqrt{d})$, we have $\supp \delta_a (f \circ \kappa^{-1}) \subset B(0,1)$. Moreover, $\delta_a (f \circ \kappa^{-1})  \in \F^s_{p,q}(\R^d)$. By Theorem \ref{localframe}(iii) we have
\begin{equation}\label{roz}
\delta_a (f \circ \kappa^{-1}) =\sum_{j\geq j_0} \sum_{\mathbf e\in E_j} \sum_{k\in \Gamma_j} \lan \delta_a (f \circ \kappa^{-1})  , f^{\mathbf e}_{(j,k)} \ran f^{\mathbf e}_{(j,k)}
\end{equation}
with unconditional convergence in $\F^s_{p,q}(\R^d)$ if $q<\infty$ and in $\F^{s-\epsilon}_{p,q}(\R^d)$ spaces  for any $\epsilon>0$ if $q=\infty$. 

Define the operator $T^2_x: L^2(B(0,3\sqrt{d})) \to L^2(\Omega_x(r))$ as in \eqref{dtp}. We can extend the domain of this operator to distributions in $\mathcal D'(\R^d)$ with compact support contained in $B(0,3\sqrt{d})$. Indeed, take any $g \in \mathcal D'(\R^d)$ with $\supp g \subset B(0,3\sqrt{d})$. Then, $\delta_{a^{-1}} g \in \mathcal D'(\R^d)$ satisfies $\supp \delta_{a^{-1}}  g \subset B(0,r)$.
Composing the distribution $\delta_{a^{-1}} g$ with the chart $\kappa$ yields a distribution in $\mathcal D'(M)$ with support in $\Omega_x(r)$. Multiplying it by $|\det g_{\kappa}|^{-1/4}$ yields a distribution $T^2_xg \in\mathcal D'(M)$, satisfying $\supp T^2_xg \subset \Omega_x(r)$. By \eqref{dtp} it follows that this definition agrees on functions. In other words, if $g$ is a function, then
\begin{equation}\label{t2x}
T^2_x g(u) = |\det g_{\kappa}|^{-1/4}(u) (\delta_{a^{-1}}g \circ \kappa)(u) \qquad\text{for }u\in \Omega_x(r).
\end{equation}
Hence, for any $g\in \mathcal D'(\R^d)$ with $\supp g \subset B(0,3\sqrt{d})$ and $\phi \in  \mathcal D(\R^d)$ with $\supp \phi \subset B(0,3\sqrt{d})$  we have
\begin{equation}\label{fact1}
\langle T^2_x g, T^2_x \phi \rangle = \langle  \delta_{a^{-1}}g \circ \kappa,  |\det g_{\kappa}|^{-1/2} \delta_{a^{-1}}\phi \circ \kappa \rangle
=  \langle  \delta_{a^{-1}}g,   \delta_{a^{-1}}\phi  \rangle
= 
\langle g, \phi \rangle.
\end{equation}

We also claim for $g \in \F^s_{p,q}(\R^d)$ with $\supp g \subset B(0,3\sqrt{d})$, we have
\begin{equation}\label{fact2}
||T^2_x g||_{\F^s_{p,q}(M)} \asymp ||T^2_xg \circ \kappa^{-1}||_{\F^s_{p,q}(\R^d)}
= ||(|\det g_{\kappa}|^{-1/4}\circ \kappa^{-1}) \delta_{a^{-1}}g ||_{\F^s_{p,q}(\R^d)}
\asymp
||g||_{\F^s_{p,q}(\R^d)}.
\end{equation}
In the first step we used \eqref{kappa}, whereas the last step uses Theorem \ref{multiplier} and the fact that the multiplier $|\det g_{\kappa}|^{-1/4}\circ \kappa^{-1}$ is bounded and bounded away from zero on $B(0,r)$.

Applying operator $T^2_x$ to both sides of \eqref{roz} and using \eqref{fact1} yields
\begin{equation}\label{recs2}
T^2_x(\delta_a (f \circ \kappa^{-1})) =\sum_{j\geq j_0} \sum_{\mathbf e\in E_j} \sum_{k\in \Gamma_j} \lan T^2_x (\delta_a (f \circ \kappa^{-1}))  , T^2_x f^{\mathbf e}_{(j,k)} \ran T^2_x f^{\mathbf e}_{(j,k)}
\end{equation}
with the same convergence as in \eqref{roz} in light of \eqref{fact2}. If $f$ is a function, then \eqref{t2x} implies that
\begin{equation}\label{t2xx}
T^2_x(\delta_a (f \circ \kappa^{-1}))(u) = \frac{f(u)}{|\det g_\kappa(u)|^{1/4}}
\qquad\text{for }u\in \Omega_x(r).
\end{equation}
Hence, to obtain \eqref{recs} for a distribution $f$ we need to apply \eqref{recs2} for $|\det g_\kappa|^{1/4} f$.

To show \eqref{perio2} we apply Theorem \ref{localframe}(iv) for $\delta_a (f \circ \kappa^{-1})  \in \F^s_{p,q}(\R^d)$
\begin{equation}\label{kwadratowa2}
\| \delta_a (f \circ \kappa^{-1})  ||_{ \F^s_{p,q}(\R^d)} \asymp
\bigg\| \bigg( \sum_{j\geq j_0} \sum_{\mathbf e\in E_j} \sum_{k\in \Gamma_{j}} \big(  2^{js} |\lan \delta_a (f \circ \kappa^{-1}) ,f^{\mathbf e}_{(j,k)}\ran|\chip_{j,k}\big)^q    \bigg)^{1/q} \bigg\|_{L^p(\R^d)}.
\end{equation}
By \eqref{fact2} and \eqref{t2xx} we have
\[
\| \delta_a (f \circ \kappa^{-1})  ||_{\F^s_{p,q}(\R^d)}
\asymp || |\det g_\kappa|^{-1/4} f ||_{ \F^s_{p,q}(M)}.
\]
By \eqref{fact1} we have
\[
\lan \delta_a (f \circ \kappa^{-1}) ,f^{\mathbf e}_{(j,k)}\ran = 
\lan |\det g_\kappa|^{-1/4} f,T^2_x f^{\mathbf e}_{(j,k)}\ran.
\]
Hence, by Lemma \ref{iso}, \eqref{tro}, and by the definition of operator $T^p_x$, we deduce \eqref{perio2}.
\end{proof}

We are now ready to show the main result of the section.

\begin{theorem}\label{triebel}
Let $M$ be a compact $d$-dimensional manifold. Let $ \F^s_{p,q}(M)$ be  a Triebel-Lizorkin space  and let $\f^s_{p,q}$ be its discrete counterpart as in Definition \ref{fs}, where  $s\in \R$, $0<p<\infty$,and  $0<q\le \infty$. Let $\mathcal W^2(M)$ be the Parseval wavelet system with smoothness parameter $m\in \N\cup \{\infty\}$ as in Definition \ref{fs}. Assume \eqref{smp}.
Then the following holds:
\begin{enumerate}[(i)]
\item
If $f\in \F^s_{p,q}(M)$, then
\[
\boldsymbol s =\{s_\psi \}\in \f^s_{p,q}(M) \qquad\text{where } s_\psi = \lan f, \psi \ran, \psi\in \mathcal W^2(M).
\]
Furthermore,
\begin{equation}\label{rownowga}
||f||_{\F^s_{p,q}(M)} \asymp ||\boldsymbol s||_{\f^s_{p,q}(M)}.
\end{equation}
\item For any $f\in \F^s_{p,q}(M)$ we have a reconstruction formula
\[
f=\sum_{\psi\in \mathcal W^2(M)}  \lan f, \psi \ran \psi,
\]
with unconditional convergence in $\F^s_{p,q}$ if $q<\infty$ and in $\F^{s-\varepsilon}_{p,q}$ for any $\varepsilon>0$ if $q=\infty$. 
\end{enumerate}
\end{theorem}

 \begin{proof}
 Fix $r_0>0$ such that $r_0/(3\sqrt{d})<r_{inj}/8$. Let $\mathcal U$ be a finite open cover of $M$ consisting of geodesic balls
\[
\mathcal U = \{ \Omega_x  :=\Omega_x(r_0/(3\sqrt{d})):  x\in X \},
\]
where $X \subset M$ is finite. 
Let $\mathcal W^2(M)$ be a Parseval frame in $L^2(M)$ consisting of $C^m$ functions localized on geodesic balls $\Omega_x=\Omega_x(r_0/3\sqrt{d})$, $x\in X$, as in Theorem \ref{pcof}.
Let $\{P_{\Omega_x}\}_{x\in X}$ be a smooth orthogonal decomposition of identity in $L^2(M)$, which is subordinate to $\mathcal U$, as Theorem \ref{rozklad}.

Let $f\in \mathcal \F^s_{p,q}(M)$. By Theorem \ref{rozklad}
\begin{equation}\label{sp}
||f||_{\F^s_{p,q}(M)} \asymp \bigg(\sum_{x\in X} \|P_{\Omega_x} f||^p_{\F^s_{p,q}(M)} \bigg)^{1/p}.
\end{equation}
By Lemma \ref{l52} and the fact that $P_{\Omega_x}$ is an orthogonal projection, we have
\begin{equation}\label{repf}
P_{\Omega_x} f= \sum_{j\geq j_0}\sum_{{\mathbf e}\in E_j} \sum_{k\in\Gamma_j}\lan P_{\Omega_x} f, P_{\Omega_x} T^{2}_x f^{\mathbf e}_{(j,k)} \ran P_{\Omega_x}T^2_x f^{\mathbf e}_{(j,k)},
\end{equation}
with unconditional convergence in $\F^s_{p,q}$ if $q<\infty$ and in $\F^{s-\epsilon}_{p,q}$ spaces  for any $\epsilon>0$ if $q=\infty$. Summing the above formula over $x\in X$ yields (ii).
Furthermore, by Lemma \ref{l52}
we have
\[
\| P_{\Omega_x} f||_{\F^s_{p,q}(M)}
\asymp
\bigg\| \bigg( \sum_{j\geq j_0} \sum_{\mathbf e\in E_j} \sum_{k\in \Gamma_{j}}  2^{jq(s+d/2)}   |\lan f, P_{\Omega_x} T^{2}_x (f^{\mathbf e}_{(j,k) }) \ran|^q \chi_{\Omega_{j ,k,x}}    \bigg)^{1/q} \bigg\|_{L^p(M)}.
\]
Summing the above formula over $x\in X$ using \eqref{sp} yields (i)
\[
\begin{aligned}
||f||_{\F^s_{p,q}(M)} & \asymp 
\bigg\| \sum_{x\in X} \bigg( \sum_{j\geq j_0} \sum_{\mathbf e\in E_j} \sum_{k\in \Gamma_{j}}  2^{jq(s+d/2)}   |\lan f, P_{\Omega_x} T^{2}_x (f^{\mathbf e}_{(j,k) }) \ran|^q \chi_{\Omega_{j ,k,x}}    \bigg)^{1/q} \bigg\|_{L^p(M)}
\\
& \asymp ||\{\lan f,\psi \ran\}_{\psi \in \mathcal W^2(M)} ||_{\f^s_{p,q}}.
\end{aligned}
\]
\end{proof}

\begin{remark}\label{synt2}
It is tempting to surmise that the sequence space $\f^s_{p,q}(M)$ characterizes coefficients of distributions in $\F^s_{p,q}(M)$ with respect to the wavelet system $\mathcal W^2(M)$. While this is true when the smoothness parameter $m$ is finite, it is actually false when $m=\infty$. This is due to the fact that wavelet system $\mathcal W^2(M)$, which is defined by localizing Meyer wavelets, is highly redundant. To describe the correct sequence space we need to add an additional decay term in the definition of $\f^s_{p,q}(M)$ as it was done in the setting of $\R^d$ in Remark \ref{synt}. We adjust Definition \ref{fs} by introducing the space 
$\f^{s,\mu}_{p,q}(M)$ with decay parameter $\mu>0$ as a collection of all sequences
\[
 \boldsymbol s= \{s_\psi\}_{\psi\in \mathcal W^2(M)}, \qquad
 \psi = P_{\Omega_x} T^2_x f^{\mathbf e}_{(j,k)}\text{ for } x\in X, j\geq j_0, k\in \Gamma_{j}, {\mathbf e}\in E_j,
 \]
with the quasi-norm
\[
\begin{aligned}
\| \boldsymbol s \|_{\f^{s,\mu}_{p,q}}  =
 \bigg\|  \bigg( \sum_{x\in X} \sum_{j\geq j_0}\sum_{{\mathbf e} \in E_j} \sum_{k\in \Gamma_{j}}
&  2^{jq(s+d/2)}  |s_\psi|^q \chi_{\Omega_{j,k,x}} \bigg)^{1/q}  \bigg\|_{L^p(M) }
 \\
& + \sup_{x\in X,\  j\geq j_0,\ \mathbf e \in E_j, \ k\in \Gamma_j \setminus \Lambda_j}
 2^{j\mu}(|2^{-j}k|_\infty+1)^\mu |s_\psi |<\infty.
\end{aligned}
\]
Then for sufficiently large $\mu>0$, the synthesis operator 
\begin{equation}\label{syn2}
\boldsymbol s=(s_\psi) \mapsto \sum_{\psi \in \mathcal W^2(M)}
s_\psi \psi
\end{equation}
is bounded from $\f^{s,\mu}_{p,q}(M)$ into $\F^s_{p,q}(M)$. This is a consequence of Remark \ref{synt} and the fact that the set $X$, which consists of centers of geodesic balls $\Omega_x$ covering a compact manifold $M$, is finite. We leave the details to the reader. As a consequence, the space $\f^{s,\mu}_{p,q}(M)$ characterizes magnitudes of coefficients of distributions in $\F^{s}_{p,q}(M)$ with respect to the wavelet system $\mathcal W^2(M)$, provided that $\mu>\max(d/p,s+d/2)$.
\end{remark}

We finish by stating a counterpart of Theorem \ref{triebel} for Besov spaces. In analogy to Definition \ref{fs} we define a discrete Besov space $\bb^s_{p,q}(M)$ as
as a set of sequences 
 \[
 \boldsymbol s= \{s_\psi\}_{\psi\in \mathcal W^2(M)}, \qquad
 \psi = P_{\Omega_x} T^2_x f^{\mathbf e}_{(j,k)}\text{ for } x\in X, j\geq j_0, k\in \Gamma_{j}, {\mathbf e}\in E_j,
 \]
such that
\[
\| \boldsymbol s \|_{\bb^s_{p,q}(M)}=
\bigg(\sum_{x\in X} \sum_{j\geq j_0}  2^{j(s+d/2-d/p)q} \sum_{\mathbf e\in E_j} \bigg( \sum_{k\in \Gamma_j}   |s^{\mathbf e}_{j,k} |^p
\bigg)^{q/p}    \bigg)^{1/q}<\infty.
\]

\begin{theorem}\label{besov}
Let $M$ be a compact $d$-dimensional manifold. Let $ \B^s_{p,q}(M)$ be  a Besov space, where  $s\in \R$, $0<p<\infty$,and  $0<q\le \infty$. Assume
\[
m> \max(s, \sigma_{p}-s) , \qquad \sigma_{p}=d\max(1/p-1,0).
\]
Then the following holds:
\begin{enumerate}[(i)]
\item
If $f\in \B^s_{p,q}(M)$, then
\[
\boldsymbol s =\{s_\psi \}\in \bb^s_{p,q}(M) \qquad\text{where } s_\psi = \lan f, \psi \ran, \psi\in \mathcal W^2(M).
\]
Furthermore,
\begin{equation}\label{rownowgaB}
||f||_{\B^s_{p,q}(M)} \asymp ||\boldsymbol s||_{\bb^s_{p,q}(M)}.
\end{equation}
\item For any $f\in \B^s_{p,q}(M)$ we have a reconstruction formula
\[
f=\sum_{\psi\in \mathcal W^2(M)}  \lan f, \psi \ran \psi,
\]
with unconditional convergence in $\B^s_{p,q}$ if $q<\infty$ and in $\B^{s-\varepsilon}_{p,q}$ for any $\varepsilon>0$ if $q=\infty$. 
\end{enumerate}
\end{theorem}

\begin{proof}
We follow along the lines of the proof of Theorem \ref{triebel}.
By Corollary \ref{hestenB} we deduce a counterpart of Theorem \ref{rozklad} for Besov spaces. 
That is,
\[
||f||_{\B^s_{p,q}(M)} \asymp \bigg(\sum_{x\in X} \|P_{\Omega_x} f||^p_{\B^s_{p,q}(M)} \bigg)^{1/p}.
\]
Since manifold $M$ is compact, the interpolation definition \eqref{ri} of Besov spaces coincides with a definition using smooth partition of unity on $M$, see \cite[(7.3.2)(8)]{Tr4} and \cite[Theorem 3]{Sk0}. Hence, we can show an analogue of Lemma \ref{l52} for Besov spaces
using Theorem \ref{localframeB} in place of Theorem \ref{localframe}.
 In particular, \eqref{repf} holds for $f\in \B^s_{p,q}$ with appropriate unconditional convergence. Moreover,
\[
\| P_{\Omega_x} f||_{ \B^s_{p,q}(M)}^q 
\asymp
\sum_{j\geq j_0}  2^{j(s+d/2-d/p)q} \sum_{\mathbf e\in E_j} \bigg( \sum_{k\in \Gamma_j}   |\lan f, P_{\Omega_x} T^{2}_x (f^{\mathbf e}_{(j,k) }) \ran |^p
\bigg)^{q/p}  .
\]
The rest of the argument is an easy adaptation of the proof of Theorem \ref{triebel}.
\end{proof}

\section{Proof of Lemma \ref{koniecprawie}}\label{SL}

In this section we give the proof of Lemma \ref{koniecprawie}, which enables us  to compute norms of localized distributions in Triebel-Lizorkin spaces using highly redundant (globally defined) Meyer wavelets on $\R^d$. Since all wavelet coefficients are needed for the reconstruction formula \eqref{ss}, it is necessary to absorb excess of frame coefficients by periodizing the formula \eqref{tld} describing the discrete Triebel-Lizorkin space $\f^s_{p,q}$. Consequently, we show a modified formula \eqref{periodic} for discrete Triebel-Lizorkin spaces, which holds for localized distributions in $\F^s_{p,q}$ spaces.

For a fixed $j\in \Z$ we define a partition of $\Z^d$ by
\begin{equation}\label{ljl}
\Lambda_{j,l}= \{k\in \Z^d: k/2^j\in 2l+[-1,1)^d \}. 
\end{equation}
Let $|l|_\infty=\max\{|l_1|\ldots,|l_d|\}$.
 If $\psi$ is a function on $\R^d$, define $\psi_{j,k}(x)=2^{jd/2}\psi(2^jx-k)$ for $j\in \Z$, $k\in \Z^d$.

\begin{lemma}\label{szacowanie}
Let $\psi \in \calS(\R^d)$.
For all $\mu>0$,  there is $C=C_{\mu,\psi}>0$ such that
for all $j\geq 0$, $l\in \Z^d$, $|l|_\infty\geq 2$ and $k\in \Lambda_{j,l}$, we have
\begin{equation}\label{Schwartz}
|\psi(2^j x-k)|\leq \frac{C}{2^{j\mu}(|l|_\infty+1)^\mu} \qquad\text{for }x\in [-2,2]^d.
\end{equation}
\end{lemma}

A straightforward proof of Lemma \ref{szacowanie} is omitted.
Using Lemma \ref{szacowanie} we deduce the following estimate for $\F^s_{p,q}(\R^d)$ spaces. 

\begin{proposition}\label{p62}
Let $\psi \in \calS(\R^d)$.
Let $0<p<\infty$, $0<q\leq \infty$ and $s\in \R$. 
For all $\mu>0$  there exists a constant $C>0$
  such that for all $f\in  \F^{s}_{p,q}(\R^d)$ with
$\supp f \subset [-1,1]^d$  and $j\geq 0$, $l\in \Z^d, |l|_\infty \geq 2$, $k\in \Lambda_{j,l}$ we have
\begin{equation}\label{Holder11}
|\lan f,\psi_{j,k }\ran |\leq \frac{ C \|f ||_{\F^s_{p,q}(\R^d)}}{2^{j\mu}(|l|_\infty+1)^\mu}.
\end{equation}
\end{proposition}

\begin{proof}
Take any $p_1>1$,  $p_1>p$. Define
\[
s_1=s-\frac{d}{p}+\frac{d}{p_1}.
\]
By  \cite[Theorem 2.7.1]{Tr3} we have a continuous embedding
\begin{equation}\label{embedd}
\F^s_{p,q}(\R^d) \hookrightarrow \F^{s_1}_{p_1, 2}(\R^d).
\end{equation}
 Fix a function
 $\eta\in C^\infty(\R^d)$ such that $\eta(x)=1$ for $x\in [-1,1]^d$ and $\supp \eta \subset [-2,2]^d$. 
 Let $m\in \N_0$ be such that $m\ge -s_1$. By the duality theorem for Triebel-Lizorkin spaces \cite[Theorem 2.11.2]{Tr3} we have
\[
(\F^{s_1}_{p_1,2}(\R^d))^*=\F^{-s_1}_{p_1',2}(\R^d),
\]
where $1/p_1+1/p_1'=1$. Combining this with \eqref{embedd} yields
 \begin{equation}\label{Holdera}
 \begin{aligned}
 |\lan f, \psi_{j,k }\ran|=|\lan f,\eta \psi_{j,k }\ran| & \leq C \|f||_{\F^{s_1}_{p_1,2}(\R^d)} \|\eta \psi_{j,k}||_{\F^{-s_1}_{p_1',2}(\R^d)}
 \\
 &\le 
 C \|f||_{\F^{s}_{p,q}(\R^d)} \|\eta \psi_{j,k}||_{\F^{m}_{p_1',2}(\R^d)}.
 \end{aligned}
 \end{equation}
Since the Triebel-Lizorkin space $\F^{m}_{p_1',2}(\R^d)$ is identified with the Sobolev space $W^m_{p_1}(\R^d)$, see \cite[Theorem 2.5.6]{Tr3} we need to control partial derivatives of $\eta \psi_{j,k}$.
Take any multi-index $\alpha \in (\N_0)^d$ such that $|\alpha|\le m$. Since the function $\eta$ is fixed by the product rule we have
\[
\begin{aligned}
\bigg(\int_{\R^d}  |\partial^\alpha(\eta \psi_{j,k})|^{p_1'}\bigg)^{1/p_1'}
&\leq 
C\sum_{|\beta| \le |\alpha|} 2^{j|\beta|} \bigg(\int_{[-2,2]^d}  |(\partial^\beta\psi)_{j,k}|^{p_1'}\bigg)^{1/p_1'}
\\
& \le
C 2^{j(m+d/2)} \sum_{|\beta| \le m} \bigg(\int_{[-2,2]^d} |\partial^\beta \psi(2^jx-k)|^{p_1'}dx \bigg)^{1/p_1'}.
\end{aligned}
\]
Applying Lemma \ref{szacowanie} to functions $\partial^\beta\psi$, $|\beta|\le m$, yields 
\[
 \|\eta \psi_{j,k}||_{\F^{m}_{p_1',2}(\R^d)} \asymp  \|\eta \psi_{j,k}||_{W^{m}_{p_1'}(\R^d)}
 \le \frac{C}{2^{j\mu}(|l|_\infty+1)^\mu}.
\]
Combining this with \eqref{Holdera} yields \eqref{Holder11}.
\end{proof}

\begin{definition}\label{qf}
Let $0<p<\infty$, $0<q\leq \infty$ and $s\in \R$.
  Let   $\psi^{\mathbf e}_{jk}$, $j\geq 0$, $k\in \Z^d, \mathbf e\in E_j$ be the multivariate Meyer wavelets as in Definition \ref{meyer}.
For a natural number $\lambda\geq 0$ we define an operator $Z_\lambda: \F^s_{p,q} \to \F^s_{p,q}$  by
\begin{equation}\label{operatorZ}
Z_\lambda(f)=\sum_{j\geq 0} \sum_{\mathbf e\in E_j}\sum_{|l|_\infty>\lambda}\sum_{ k\in \Lambda_{j,l}} \lan f,\psi^{\mathbf e}_{j, k}\ran \psi^{\mathbf e}_{j, k } \qquad\text{for }f\in \F^s_{p,q}(\R^d).
\end{equation}
We define a $\mathrm{q}$-function of $f$ as
\begin{equation}\label{qquad}
\mathrm{q}(f)=\Big(\sum_{j\geq 0} \sum_{\mathbf e\in E_j} \sum_{ k\in \Z^d} \big( 2^{js}  \chi_{j,k} |\lan f,\psi^{\mathbf e}_{j,k}\ran| \big)^q \Big)^{1/q},
\end{equation}
where $\chi_{j,k}(x)=2^{jd/2} \chi_I(2^jx-k)$, $I=[0,1]^d$.
In the case $q=\infty$ the above definition involves $\ell^\infty$ norm.
 \end{definition}

 The operator $Z_\lambda$ is well defined and bounded by \cite[Theorem 3.12]{Trtom3}.
In addition, the quasi norm (or norm) in Triebel-Lizorkin spaces is equivalent with $L^p$ norm of $\mathrm q$-function, i.e.
\begin{equation}\label{normadyskretny}
\|f||_{ \F^s_{p,q}(\R^d)}^p \asymp \int_{\R^d}   {\mathrm q}^{p}(f)
=
\int_{\R^d} \Big( \sum_{j\geq 0} \sum_{\mathbf e\in E_j} \sum_{ k\in \Z^d} \big( 2^{js}  \chi_{j,k}(x) |\lan f,\psi^{\mathbf e}_{j,k}\ran| \big)^q \Big)^{p/q} dx.
\end{equation}

\begin{proposition}\label{klucz}
Let $0<p<\infty$, $0<q\leq \infty$, and $s\in \R$. Let  $f\in \F^s_{p,q}(\R^d)$ be such that $\supp f\subset [-1,1]^d.$ 
Let $\lambda\geq 2$.  Then for any $\mu>0$ there is $C_\mu$ such that
\begin{equation}\label{norma1}
\|Z_\lambda f||_{\F^s_{p,q}(\R^d)} \leq C_\mu \lambda^{-\mu}\|f||_{\F^s_{p,q}(\R^d)}.
\end{equation}
Consequently, for sufficiently large $\lambda\geq 2$ we have
\begin{equation}\label{norma2}
\begin{aligned}
\|f||_{ \F^s_{p,q}(\R^d)} &\asymp \bigg(\int_{\R^d}   {\mathrm q}^{p}(f-Z_\lambda f) \bigg)^{1/p}
\\
&= \bigg(\int_{\R^d} \Big( \sum_{j\geq 0} \sum_{\mathbf e\in E_j} \sum_{|l|_\infty \leq \lambda}\sum_{ k\in \Lambda_{j,l}} \big( 2^{js}  \chi_{j,k}(x) |\lan f,\psi^{\mathbf e}_{j,k}\ran| \big)^q \Big)^{p/q} dx \bigg)^{1/p}
\end{aligned}
\end{equation}
\end{proposition}

\begin{proof}
For simplicity we assume that $q<\infty$; the case $q=\infty$ follows by easy modifications.
Since the set $\{\psi^{\mathbf e}_{j,k}: j\geq 0, \mathbf e \in E_j, k\in \Z^d\}$ is an orthonormal basis of $L^2(\R^d)$ we have
\[
{\mathrm q}^q(Z_\lambda f)=\sum_{j\geq 0} \sum_{\mathbf e\in E_j} \sum_{|l|_\infty>\lambda}\sum_{ k\in \Lambda_{j,l}} \big( 2^{js}  \chi_{j,k} |\lan f,\psi^{\mathbf e}_{j,k}\ran| \big)^q.
\]
By scaling we can assume that  $\|f||_{\F^s_{p,q}(\R^d)}=1$. By \eqref{Holder11} we have
\[
{\mathrm q}^q(Z_\lambda f)\leq C \sum_{j\geq 0}  \sum_{|l|_\infty>\lambda}\sum_{ k\in \Lambda_{j,l}} \bigg(
\frac{2^{j(s-\mu)} }{(1+|l|_\infty)^\mu} \chi_{j,k}  \bigg)^q.
\]
Take any $x\in 2\beta+[-1,1)^d$, $\beta\in \Z^d$. If $k\in \Lambda_{j,l}$, then
\[
x-2^{-j}k \in 2(\beta-l)+2[-1,1)^d.
\]
Hence, if $\chi_{j,k}(x) \ne 0$, then $x-2^{-j}k  \in 2^{-j}[0,1]^d$, and hence $l=\beta$.
In particular, if $|\beta|_\infty \leq \lambda$, then
\[
\mathrm{q}(Z_\lambda f)(x)=0.
\]
On the other hand, if $|\beta|_\infty > \lambda$, then
\[
\mathrm{q}^q(Z_\lambda f)(x) \leq C\sum_{j\geq 0}  \bigg( 2^{j(s-\mu)} 2^{jd/2} 
\frac{1}{(1+|\beta|_\infty)^\mu} \bigg)^q.
\]
Choose sufficiently large $\mu$ such that $\delta=s-\mu+d/2<0$. Then,
\[
\mathrm q(Z_\lambda f)(x) \leq C\frac{1}{(1+|\beta|_\infty)^\mu}  \Big( \sum_{j\geq 0}  \big( 2^{j(s-\mu)} 2^{jd/2} \big)^q \Big)^{1/q}\leq C' \frac{1}{(1+|\beta|_\infty)^\mu}
\]
Then,
\[
\int_{\R^d} (\mathrm q(Z_\lambda f)(x) )^{p} dx \leq 2^d (C')^p \sum_{|\beta|_\infty > \lambda} \frac{1}{(1+|\beta|_\infty)^{p\mu}}  .
\]
If we assume additionally  that $p\mu>d$, then
\begin{equation}\label{rho}
\int_{\R^d} (\mathrm q(Z_\lambda f)(x) )^{p} dx \leq C'' \lambda^{-p\mu+d} .
\end{equation}
Hence, by \eqref{normadyskretny} we have
\begin{equation}\label{lambda}
\|Z_\lambda f||_{ \F^s_{p,q}(\R^d)}^p \asymp \int_{\R^d}  \big( \mathrm q(Z_\lambda f) \big)^{p}\leq C''  \lambda^{-p\mu+d} \|f||_{\F^s_{p,q}(\R^d)}^p.
\end{equation}
Since $\mu$ is arbitrarily large, we deduce \eqref{norma1}.
By the triangle inequality (with a constant) for $\F^s_{p,q}$ space we deduce that the norm of $f$ is comparable with the norm of $f-Z_\lambda f$ for sufficiently large $\lambda$, which shows \eqref{norma2}.
\end{proof}

We are now ready to complete the proof of Lemma \ref{koniecprawie}.

\begin{proof}[Proof of Lemma \ref{koniecprawie}]
Take $\xi=4\lambda+2$, where $\lambda \ge 2$ is sufficiently large as in Proposition \ref{klucz}. We shall show that \eqref{periodic} holds with $\lambda$ replaced by $\xi$. 
By \eqref{ljl} we have
\[
\bigcup_{|l|_\infty\leq \lambda} \Lambda_{j,l}=\{k\in \Z^d: k/2^j\in [-2\lambda-1,2\lambda+1)^d \}=\Lambda_j.
\] 
By \eqref{norma2}, we have
 \begin{equation}\label{norma4}
\|f||_{ \F^s_{p,q}(\R^d))}^p \asymp \int_{\R^d} \Big( \sum_{j\geq 0} \sum_{\mathbf e\in E_j} \sum_{ k\in \Lambda_{j}} \big( 2^{js}  \chi_{j,k}(x) |\lan f,\psi^{\mathbf e}_{j,k}\ran| \big)^q \Big)^{p/q} dx.
\end{equation}
Hence, we automatically have
\begin{equation}
\label{clear}
\| f||_{ \F^s_{p,q}(\R^d)}^p\leq C
\int_{\R^d} \Big( \sum_{j\geq 0} \sum_{\mathbf e\in E_j} \sum_{ k\in \Lambda_{j}} \sum_{l\in \Z^d} \big( 2^{js}  \chi_{j,k}(x) |\lan f,\psi^{\mathbf e}_{j,k+2^jl\xi}\ran| \big)^q \Big)^{p/q} dx.
\end{equation}

The reverse inequality is a consequence of Proposition \ref{p62}. Indeed, the integration on the right hand side of \eqref{clear} to can be restricted to $[-2\lambda-1,2\lambda+1)^d$ since $\supp \chi_{j,k}= 2^{-j}(I+k)$.
Hence, taking into consideration \eqref{norma4} and \eqref{clear} it is sufficient to prove that there is $C>0$ such that for $x\in [-2\lambda-1,2\lambda+1)^d$ 
\begin{equation}\label{norma5}
\Big( \sum_{j\geq 0} \sum_{\mathbf e\in E_j} \sum_{ k\in \Lambda_{j}} \sum_{l\neq 0} \big( 2^{js}  \chi_{j,k}(x) |\lan f,\psi^{\mathbf e}_{j,k+2^jl\xi}\ran| \big)^q \Big)^{1/q}
\leq
C \| f||_{ \F^s_{p,q}(\R^d)}.
\end{equation}
Take any $k\in \Lambda_j$. Then, $k\in \Lambda_{j,\tilde l}$ for some  $|\tilde{l}|_\infty \leq \lambda$. Hence,
\[
2^{-j}(k+2^jl\xi) \in 2(\tilde{l}+l\xi/2)+[-1,1)^d=  \Lambda_{j,\tilde l+l\xi/2}.
\]
If $l\in \Z^d$ and $l \neq 0$, then $|\tilde{l}+l\xi/2|_\infty \ge 2$, by \eqref{Holder11}  we have
\[
 |\lan f,\psi^{\mathbf e}_{j,k+2^jl\xi}\ran| \leq  \frac{ C \|f ||_{\F^s_{p,q}(\R^d)}}{2^{j\mu}(|\tilde{l}+l\xi/2|_\infty+1)^\mu}.
 \]
 Hence, taking $\mu>\max(d,s+d/2)$ yields a constant $C'$ such that
 \[
\sum_{l\ne 0}  |\lan f,\psi^{\mathbf e}_{j,k+2^jl\xi}\ran|^q \leq  C' 2^{-jq\mu}  \|f ||_{\F^s_{p,q}(\R^d)}^q. 
 \]
 Therefore,
 \[
 \sum_{j\geq 0} \sum_{\mathbf e\in E_j} \sum_{ k\in \Lambda_{j}} \sum_{l\neq 0} \big( 2^{js}  \chi_{j,k}(x) |\lan f,\psi^{\mathbf e}_{j,k+2^jl\xi}\ran| \big)^q \le 
 C'  2^d  \|f ||_{\F^s_{p,q}(\R^d) }^q \sum_{j\geq 0} 2^{j(s+d/2-\mu)q}.
 \]
This proves \eqref{norma5}. 
\end{proof}


\begin{thebibliography}{99}


\bibitem{AWW} P. Auscher, G. Weiss, M. V. Wickerhauser,
{\it Local sine and cosine bases of Coifman and Meyer and the construction of smooth wavelets.} Wavelets, 237--256,
Wavelet Anal. Appl., 2, Academic Press, Boston, MA, 1992.

\bibitem{BD}
M. Bownik, K. Dziedziul,
{\it Smooth orthogonal projections on sphere}. Const. Approx. {\bf 41} (2015), 23--48.

\bibitem{BDK}
M. Bownik, K. Dziedziul, A. Kamont,
{\it Smooth orthogonal projections on Riemannian manifold}.  Potential Anal. (to appear).

\bibitem{cdosz}
P. G. Casazza, S. J. Dilworth, E. Odell, Th. Schlumprecht, A. Zs\'ak,
{\it Coefficient quantization for frames in Banach spaces}.
J. Math. Anal. Appl. {\bf 348} (2008), no. 1, 66--86. 

\bibitem{Casazza} P.G. Casazza,  D. Han,  D. R. Larson,  {\it Frames for Banach spaces. The functional and harmonic analysis of wavelets and frames} (San Antonio, TX, 1999), 149--182, Contemp. Math., 247, Amer. Math. Soc., Providence, RI, 1999.
 
\bibitem{Ch}
I. Chavel, {\it Isoperimetric inequalities. Differential geometric and analytic perspectives}.
Cambridge Tracts in Mathematics, 145. Cambridge University Press, Cambridge, 2001.

\bibitem{Ch2}
I. Chavel, {\it Riemannian geometry.
A modern introduction}. Second edition. Cambridge Studies in Advanced Mathematics, 98. Cambridge University Press, Cambridge, 2006.

\bibitem{CF1} Z. Ciesielski, T. Figiel,
{\it Spline approximation and Besov spaces on compact manifolds.} Studia Math. {\bf 75} (1982), no. 1, 13--36.

\bibitem{CF2} Z. Ciesielski, T. Figiel,
{\it Spline bases in classical function spaces on compact $C^\infty$ manifolds. I.}
Studia Math. {\bf 76} (1983), no. 1, 1--58.

\bibitem{CF3} Z. Ciesielski, T. Figiel,
{\it  Spline bases in classical function spaces on compact $C^\infty$ manifolds. II.}
 Studia Math. {\bf 76} (1983), no. 2, 95--136.
 

\bibitem{CM} R. Coifman, Y. Meyer,
{\it Remarques sur l'analyse de Fourier \`a fen\^ etre.}
C. R. Acad. Sci. Paris S\' er. I Math. {\bf 312} (1991), no. 3, 259--261.

\bibitem{CKP} T. Coulhon, G. Kerkyacharian, P. Petrushev, 
{\it Heat kernel generated frames in the setting of Dirichlet spaces}.
J. Fourier Anal. Appl. {\bf 18} (2012), no. 5, 995--1066. 


\bibitem{DaSch}
W. Dahmen, R. Schneider,
{\it Wavelets on manifolds. I. Construction and domain decomposition.}
SIAM J. Math. Anal. {\bf 31} (1999), no. 1, 184--230.

\bibitem{Dau}
I. Daubechies,
{\it Ten lectures on wavelets}. CBMS-NSF Regional Conference Series in Applied Mathematics, Vol. 61, Society for Industrial and Applied Mathematics, Philadelphia, PA (1992).


\bibitem{FHP}
H. Feichtinger, H. F\"uhr, I. Pesenson, 
{\it Geometric space-frequency analysis on manifolds.}
J. Fourier Anal. Appl. {\bf 22} (2016), no. 6, 1294--1355. 

\bibitem{FW}
T. Figiel, P. Wojtaszczyk,
{\it Special bases in function spaces.} Handbook of the geometry of Banach spaces, Vol. I, 561--597, North-Holland, Amsterdam, 2001.


\bibitem{FJ}
M. Frazier, B. Jawerth,
{\it A Discrete transform and decomposition of distribution spaces.}
J. Funct. Anal. {\bf 93}, (1990), 34--170.



\bibitem{GM1}
D. Geller, A. Mayeli, 
{\it Nearly tight frames and space-frequency analysis on compact manifolds}.
 Math. Z. {\bf 263} (2009), no. 2, 235--264.
 
\bibitem{GM2}
D. Geller, A. Mayeli, 
{\it Besov spaces and frames on compact manifolds.}
Indiana Univ. Math. J. {\bf 58} (2009), no. 5, 2003--2042. 

\bibitem{GP}
D. Geller, I. Pesenson, 
{\it Band-limited localized Parseval frames and Besov spaces on compact homogeneous manifolds}. J. Geom. Anal. {\bf 21} (2011), 334--371. 

\bibitem{Grosse}
N. Gro\ss e, C. Schneider, 
{\it Sobolev spaces on Riemannian manifolds with bounded geometry: general coordinates and traces}.
Math. Nachr. {\bf 286} (2013), no. 16, 1586--1613. 

\bibitem{Grochenig} K. Grochenig, {\it Describing functions: Atomic decompositions versus frames}.  Monatsh. Math. {\bf 112} (1991), 1--41.


\bibitem{He}
E. Hebey, {\it Nonlinear analysis on manifolds: Sobolev spaces and inequalities}. Courant Lecture Notes in Mathematics, 5. American Mathematical Society, Providence, RI, 1999.

\bibitem{HW} E. Hern\' andez, G. Weiss,
{\it A first course on wavelets.}
Studies in Advanced Mathematics. CRC Press, Boca Raton, FL, 1996.

\bibitem{Hor}
L. H\"ormander, 
{\it The analysis of linear partial differential operators. I.
Distribution theory and Fourier analysis.} Reprint of the second (1990) edition. Classics in Mathematics. Springer-Verlag, Berlin, 2003.

\bibitem{ks}
A. Kunoth, J. Sahner,
{\it Wavelets on manifolds: an optimized construction.}
Math. Comp. {\bf 75} (2006), no. 255, 1319--1349.

\bibitem{Lew}
J. S. Lew,
{\it Extension of a local diffeomorphism}.
Arch. Rational Mech. Anal. {\bf 26} (1967), 400--402. 


\bibitem{Me}
Y.~Meyer, {\it Wavelets and operators}.
Cambridge University Press, Cambridge, 1992.

\bibitem{Pal1}
R. S. Palais,
{\it Natural operations on differential forms}.
Trans. Amer. Math. Soc. {\bf 92} (1959), 125--141.

\bibitem{Pal2}
R. S. Palais,
{\it Extending diffeomorphisms}.
Proc. Amer. Math. Soc. {\bf 11} (1960), 274--277. 


\bibitem{Pe2}
I. Pesenson, {\it Paley-Wiener-Schwartz nearly Parseval frames on noncompact symmetric spaces}. Commutative and noncommutative harmonic analysis and applications, 55--71, Contemp. Math., 603, Amer. Math. Soc., Providence, RI, 2013. 

\bibitem{Pe3}
I. Pesenson,
{\it Parseval space-frequency localized frames on sub-Riemannian compact homogeneous manifolds}. Frames and other bases in abstract and function spaces, 413--433,
Appl. Numer. Harmon. Anal., Birkhäuser/Springer, Cham, 2017. 

\bibitem{Sh}
M. A. Shubin, 
{\it Spectral theory of elliptic operators on noncompact manifolds}.
M\'ethodes semi-classiques, Vol. 1 (Nantes, 1991).
Ast\'erisque No. 207 (1992), 5, 35--108. 

\bibitem{Sk0}
L. Skrzypczak, 
{\it Atomic decompositions on manifolds with bounded geometry}.
Forum Math. {\bf 10} (1998), no. 1, 19--38. 

\bibitem{Sk}
L. Skrzypczak, 
{\it Wavelet frames, Sobolev embeddings and negative spectrum of Schrödinger operators on manifolds with bounded geometry}.
J. Fourier Anal. Appl. {\bf 14} (2008), no. 3, 415--442. 



\bibitem{Tr3}
H. Triebel, 
{\it  Theory of function spaces}.
Monographs in Mathematics, 84. Birkh\"auser Verlag, Basel, 1983.

\bibitem{Tr1}
H. Triebel,
{\it Spaces of Besov-Hardy-Sobolev type on complete Riemannian manifolds}.
Ark. Mat. {\bf 24} (1986), no. 2, 299--337. 

\bibitem{Tr2}
H. Triebel,
{\it Characterizations of function spaces on a complete Riemannian manifold with bounded geometry.}
Math. Nachr. {\bf 130} (1987), 321--346. 

\bibitem{Tr4}
H. Triebel, 
{\it  Theory of function spaces. II}.
Monographs in Mathematics, 84. Birkh\"auser Verlag, Basel, 1992.

\bibitem{Trtom3}
H. Triebel, 
{\it  Theory of function spaces. III}.
Monographs in Mathematics, 84. Birkh\"auser Verlag, Basel, 2006.

\bibitem{Tr5}
H. Triebel, 
{\it Function spaces and wavelets on domains}.
EMS Tracts in Mathematics, 7. European Mathematical Society (EMS), Z\"urich, 2008.


\bibitem{Wo}
P.~Wojtaszczyk, {\it A mathematical introduction to wavelets}.
Cambridge University Press, Cambridge, 1997.


\end{thebibliography}
\end{document}